\documentclass{amsart}
\pdfoutput=1

\usepackage{amssymb}
\usepackage{microtype}
\usepackage[numbers]{natbib}
\usepackage{enumerate}
\usepackage[all,tips]{xy}
\SelectTips{cm}{10}

\usepackage{aliascnt}

\usepackage{hyperref}
\hypersetup{
  pdftitle={Semilinear mixed problems on Hilbert complexes and their
    numerical approximation},
  pdfauthor={Michael Holst and Ari Stern},
  pdfsubject={MSC 2010: Primary 65N30; Secondary 35J91, 47H30},
  bookmarksopen=false
}

        \theoremstyle{plain} 
        \newtheorem{theorem}{Theorem}[section]
        \newaliascnt{lemma}{theorem}
        \newtheorem{lemma}[lemma]{Lemma}
        \aliascntresetthe{lemma}
        \newaliascnt{proposition}{theorem}
        
        \aliascntresetthe{proposition}
        \newaliascnt{corollary}{theorem}
        \newtheorem{corollary}[corollary]{Corollary}
        \aliascntresetthe{corollary}

        \theoremstyle{definition}
        \newtheorem{definition}[theorem]{Definition}
        \newtheorem{example}[theorem]{Example}

        \theoremstyle{remark}
        \newtheorem{remark}{Remark}
        \newtheorem*{acknowledgments}{Acknowledgments}

\begin{document}

\title[Semilinear mixed problems on Hilbert complexes]{Semilinear
  mixed problems on Hilbert complexes\\
and their numerical approximation}

\author{Michael Holst}
\address{Department of Mathematics\\
University of California, San Diego\\
9500 Gilman Dr \#0112\\
La Jolla CA 92093-0112}
\email{\{mholst,astern\}@math.ucsd.edu}

\author{Ari Stern}

\subjclass[2010]{Primary: 65N30; Secondary: 35J91, 47H30}

\begin{abstract}
  Arnold, Falk, and Winther recently showed
  [\emph{Bull. Amer. Math. Soc.} \textbf{47} (2010), 281--354] that
  linear, mixed variational problems, and their numerical
  approximation by mixed finite element methods, can be studied using
  the powerful, abstract language of Hilbert complexes.  In another
  recent article [arXiv:1005.4455], we extended the
  Arnold--Falk--Winther framework by analyzing variational crimes (a
  la Strang) on Hilbert complexes.  In particular, this gave a
  treatment of finite element exterior calculus on manifolds,
  generalizing techniques from surface finite element methods and
  recovering earlier \emph{a priori} estimates for the
  Laplace--Beltrami operator on $2$- and $3$-surfaces, due to Dziuk
  [\emph{Lecture Notes in Math.}, vol.~1357 (1988), 142--155] and
  later Demlow [\emph{SIAM J. Numer. Anal.}, \textbf{47} (2009),
  805--827], as special cases.  In the present article, we extend the
  Hilbert complex framework in a second distinct direction: to the
  study of semilinear mixed problems.  We do this, first, by
  introducing an operator-theoretic reformulation of the linear mixed
  problem, so that the semilinear problem can be expressed as an
  abstract Hammerstein equation.  This allows us to obtain, for
  semilinear problems, \emph{a priori} solution estimates and error
  estimates that reduce to the Arnold--Falk--Winther results in the
  linear case.  We also consider the impact of variational crimes,
  extending the results of our previous article to these semilinear
  problems.  As an immediate application, this new framework allows
  for mixed finite element methods to be applied to semilinear
  problems on surfaces.
\end{abstract}

\date{August 9, 2011}

\maketitle

\clearpage

\tableofcontents

\section{Introduction}
\label{sec:intro}

The goal of this paper is to extend the abstract Hilbert complex
framework of \citet{ArFaWi2010}---which they introduced to analyze
certain linear mixed variational problems and their numerical
approximation by mixed finite elements---to a class of
\emph{semilinear} mixed variational problems.  Additionally, we aim to
analyze variational crimes in this semilinear setting, extending our
earlier analysis of the linear case in \citet{HoSt2010}.

\subsection{Background}

\citet{BrLe1992} originally studied Hilbert complexes as a way to
generalize certain properties of elliptic complexes, particularly the
Hodge decomposition and other aspects of Hodge theory.  More recently,
\citet{ArFaWi2010} showed that Hilbert complexes are also a convenient
abstract setting for mixed variational problems and their numerical
approximation by mixed finite element methods, providing the
foundation of a framework called \emph{finite element exterior
  calculus} (see also~\citep{ArFaWi2006}).  This line of research is
the culmination of several decades of work on mixed finite element
methods, which have long been used with great success in computational
electromagnetics, and which were more recently discovered to have
surprising connections with the calculus of exterior differential
forms, including de~Rham cohomology and Hodge
theory~\citep{Bossavit1988,Nedelec1980,Nedelec1986,GrKo2004}. For this
reason, Hilbert complexes are a natural fit for abstract methods of
this type.

Another recent development in this area has been the analysis of
``variational crimes'' on Hilbert complexes (\citet{HoSt2010}).  By
analogy with Strang's lemmas for variational crimes on Hilbert spaces,
this work extended the estimates of \citet{ArFaWi2010} to problems
where certain conditions on the discretization have been violated.
This framework also allowed for a generalization of several results in
the field of \emph{surface finite element methods}, where a curved
domain is not triangulated exactly, but is only approximated by, e.g.,
piecewise linear or isoparametric elements.  This research area was
initiated with the 1988 article of \citet{Dziuk1988} (see
also~\citet{Nedelec1976}), with growing activity in the 1990s
\citep{Dziuk1991,DeDz1995} and a substantial expansion beginning
around
2001~\citep{Holst2001,Christiansen2002,DeDz2003,DeDzEl2005,DzHu2006,DzEl2007,DeDz2007,Demlow2009}.

Our main motivation for extending the estimates of \citet{ArFaWi2010}
and of \citet{HoSt2010}, from linear to semilinear problems, is to
enable the use of finite element exterior calculus for nonlinear
problems on hypersurfaces, allowing for a complete analysis of the
additional errors due to nonlinearity, as well as those due to surface
approximation.

\subsection{Organization of the paper}

The remainder of the article is structured as follows.  In
\autoref{sec:hilbert} we give a quick overview of abstract Hilbert
complexes and their properties, before introducing the Hodge Laplacian
and the linear mixed problem associated with it.  We then discuss the
numerical approximation of solutions to this problem, summarizing some
of the key results of \citet{ArFaWi2010} on approximation by
subcomplexes, and those of \citet{HoSt2010} on variational crimes.  In
\autoref{sec:semilinear}, we introduce an alternative,
operator-theoretic formalism for the linear problem, which---while
equivalent to the mixed variational formulation---allows for a more
natural extension to semilinear problems, due to its monotonicity
properties.  We then introduce a class of semilinear problems---which
can be expressed in the form of certain nonlinear operator equations,
called abstract Hammerstein equations---prove the well-posedness of
these problems, and establish solution estimates under various
assumptions on the nonlinear part.  In \autoref{sec:approx}, we extend
the \emph{a priori} error estimates of \citet{ArFaWi2010} from linear
problems to the semilinear problems introduced in
\autoref{sec:semilinear}, including improved estimates subject to
additional compactness and continuity assumptions.  Finally, we
generalize the linear variational crimes framework of \citep{HoSt2010}
to this class of semilinear problems.  These last results allow the
linear \emph{a priori} estimates, established in \citep{HoSt2010} for
surface finite element methods using differential forms on
hypersurfaces, to be extended to semilinear problems.

\section{Review of Hilbert complexes and linear mixed problems}
\label{sec:hilbert}

We begin, in this section, by quickly recalling the basic objects of
interest---Hilbert complexes and the abstract Hodge Laplacian---along
with the solution theory for linear mixed problems in this setting.
This provides the background and preparation for semilinear problems,
which will be discussed in the subsequent sections.  The treatment of
this background material will be necessarily brief; we will primarily
follow the approach of \citet{ArFaWi2010}, to which the interested
reader should refer for more detail.\footnote{This is largely a
  condensed version of the background material given in \citet[Section
  2]{HoSt2010}, from which we quote freely.  We include it here in the
  interest of keeping the present paper self-contained, since the
  semilinear theory will depend, to a large degree, on several
  properties and results that have recently been established for the
  linear problem.}  At the end of the section, we will also summarize
the results from \citet{HoSt2010}, analyzing variational crimes for
the linear problem, in preparation for extending these results to the
semilinear case.

\subsection{Basic definitions} First, let us introduce the objects of
study, Hilbert complexes, and their morphisms.

\begin{definition}
  A \emph{Hilbert complex} $ \left( W , \mathrm{d} \right) $ consists
  of a sequence of Hilbert spaces $ W ^k $, along with closed,
  densely-defined linear maps $ \mathrm{d} ^k \colon V ^k \subset W ^k
  \rightarrow V ^{ k + 1 } \subset W ^{ k + 1 } $, possibly unbounded,
  such that $ \mathrm{d} ^k \circ \mathrm{d} ^{k-1} = 0 $ for each
  $k$.
  \begin{equation*}
    \xymatrix{  \cdots \ar[r] & V ^{ k - 1 } \ar[r]^-{ \mathrm{d} ^{k-1} }
      & V ^k \ar[r] ^-{ \mathrm{d} ^k } & V ^{ k + 1 } \ar[r] & \cdots }
  \end{equation*}
  This Hilbert complex is said to be \emph{bounded} if $ \mathrm{d} ^k
  $ is a bounded linear map from $ W ^k $ to $ W ^{ k + 1 } $ for each
  $k$, i.e., $ \left( W , \mathrm{d} \right) $ is a cochain complex in
  the category of Hilbert spaces.  It is said to be \emph{closed} if
  the image $ \mathrm{d} ^k V ^k $ is closed in $ W ^{ k + 1 } $ for
  each $k$.
\end{definition}

\begin{definition}
  Given two Hilbert complexes, $ \left( W, \mathrm{d} \right) $ and $
  \left( W ^\prime , \mathrm{d} ^\prime \right) $, a \emph{morphism of
    Hilbert complexes} $ f \colon W \rightarrow W ^\prime $ consists
  of a sequence of bounded linear maps $ f ^k \colon W ^k \rightarrow
  W ^{\prime k} $ such that $ f ^k V ^k \subset V ^{ \prime k } $ and
  $ \mathrm{d} ^{\prime k} f ^k = f ^{ k + 1 } \mathrm{d} ^k $ for
  each $k$.  That is, the following diagram commutes:
  \begin{equation*}
    \xymatrix@=3em{
      \cdots \ar[r] & V ^k \ar[d]^{f ^k } \ar[r]^{\mathrm{d}^k} &
      V ^{k+1} \ar[d]^{f ^{k+1}} \ar[r] & \cdots \\
      \cdots \ar[r] & V ^{\prime k} \ar[r]^{\mathrm{d} ^{\prime k}}
      & V ^{\prime k+1} \ar[r] & \cdots 
    }
  \end{equation*}
\end{definition}

By analogy with cochain complexes, it is possible to define notions of
cocycles, coboundaries, and harmonic forms for Hilbert complexes.
(This also gives rise to a cohomology theory for Hilbert complexes.)

\begin{definition}
  Given a Hilbert complex $ \left( W, \mathrm{d} \right) $, the space
  of \emph{$k$-cocycles} is the kernel $ \mathfrak{Z} ^k = \ker
  \mathrm{d} ^k $, the space of \emph{$k$-coboundaries} is the image $
  \mathfrak{B} ^k = \mathrm{d} ^{ k - 1 } V ^{ k - 1 } $, and the
  \emph{$k$th harmonic space} is the intersection $ \mathfrak{H} ^k =
  \mathfrak{Z} ^k \cap \mathfrak{B} ^{k \perp} $.
\end{definition}

In general, the differentials $ \mathrm{d} ^k $ of a Hilbert complex
may be unbounded linear maps.  However, given an arbitrary Hilbert
complex $ \left( W, \mathrm{d} \right) $, it is always possible to
construct a bounded complex having the same domains and maps, as
follows.

\begin{definition}
  Given a Hilbert complex $ \left( W, \mathrm{d} \right) $, the
  \emph{domain complex} $ \left( V, \mathrm{d} \right) $ consists of
  the domains $ V ^k \subset W ^k $, endowed with the graph inner
  product
\begin{equation*}
  \left\langle u, v \right\rangle _{ V ^k } = \left\langle u , v
  \right\rangle _{ W ^k } + {\langle \mathrm{d} ^k u, \mathrm{d}
    ^k v \rangle} _{ W ^{ k + 1 } } .
\end{equation*}
\end{definition}

\begin{remark}
  Since $ \mathrm{d} ^k $ is a closed map, each $ V ^k $ is closed
  with respect to the norm induced by the graph inner product.  Also,
  each map $ \mathrm{d} ^k $ is bounded, since
  \begin{equation*}
    {\lVert \mathrm{d} ^k v \rVert} _{V^{k+1}} = {\lVert \mathrm{d}
      ^k v \rVert} _{W^{k+1}} \leq \left\lVert v \right\rVert _{ W ^k
    } + {\lVert \mathrm{d} ^k v \rVert} _{W^{k+1}} = \left\lVert
      v \right\rVert _{ V ^k } .
  \end{equation*}
  Thus, the domain complex is a bounded Hilbert complex; moreover, it
  is a closed complex if and only if $ \left( W, \mathrm{d} \right) $
  is closed.
\end{remark}

\begin{example}
  Perhaps the most important example of a Hilbert complex arises from
  the de~Rham complex $ \left( \Omega (M), \mathrm{d} \right) $ of
  smooth differential forms on an oriented, compact, Riemannian
  manifold $M$, where $ \mathrm{d} $ is the exterior derivative.
  Given two smooth $k$-forms $u, v \in \Omega ^k (M) $, the $ L ^2
  $-inner product is defined by
  \begin{equation*}
    \left\langle u, v \right\rangle _{ L ^2 \Omega (M) } = \int _M u
    \wedge \star v = \int _M \left\langle \! \left\langle u, v
      \right\rangle \! \right\rangle  \mu ,
  \end{equation*}
  where $ \star $ is the Hodge star operator associated to the
  Riemannian metric, $ \left\langle \! \left\langle \cdot, \cdot
    \right\rangle \! \right\rangle $ is the metric itself, and $\mu$
  is the Riemannian volume form.  The Hilbert space $ L ^2 \Omega ^k
  (M) $ is then defined, for each $k$, to be the completion of $
  \Omega ^k (M) $ with respect to the $ L ^2 $-inner product.  One can
  also define weak exterior derivatives $ \mathrm{d} ^k \colon H
  \Omega ^k (M) \subset L ^2 \Omega ^k (M) \rightarrow H \Omega ^{ k +
    1 } (M) \subset L ^2 \Omega ^{ k + 1 } (M) $; the domain complex $
  \left( H \Omega (M) , \mathrm{d} \right) $, with the graph inner
  product
  \begin{equation*}
    \left\langle u , v \right\rangle _{ H \Omega (M) } = \left\langle
      u, v \right\rangle _{ L ^2 \Omega (M) } + \left\langle
      \mathrm{d} u, \mathrm{d} v \right\rangle _{ L ^2 \Omega (M) } ,
  \end{equation*}
  is analogous to a Sobolev space of differential forms.  (For
  example, in $ \mathbb{R}^3 $, the domain complex corresponds to the
  spaces $ H ^1 $, $ H \left( \operatorname{curl} \right) $, and $ H
  \left( \operatorname{div} \right) $.)  Finally, we mention the fact
  that both the $ L ^2 $- and $H$-de~Rham complexes are closed.  For a
  detailed treatment of these complexes, and their many applications,
  see \citet{ArFaWi2010}.
\end{example}

For the remainder of the paper, we will follow the simplified notation
used by \citet{ArFaWi2010}: the $W$-inner product and norm will be
written simply as $ \left\langle \cdot , \cdot \right\rangle $ and $
\left\lVert \cdot \right\rVert $, without subscripts, while the
$V$-inner product and norm will be written explicitly as $
\left\langle \cdot , \cdot \right\rangle _V $ and $ \left\lVert \cdot
\right\rVert _V $.

\subsection{Hodge decomposition and the Poincar\'e inequality}

For $ L ^2 $ differential forms, the Hodge decomposition states that
any $k$-form can be written as a direct sum of exact, coexact, and
harmonic components.  (In $\mathbb{R}^3$, this corresponds to the
Helmholtz decomposition of vector fields.)  In fact, this can be
generalized to give a Hodge decomposition for arbitrary Hilbert
complexes; this immediately gives rise to an abstract version of the
Poincar\'e inequality, which is crucial to much of the analysis in
\citet{ArFaWi2010}.

Following \citet{BrLe1992}, we can decompose each space $ W ^k $ in
terms of orthogonal subspaces,
\begin{equation*}
  W ^k = \mathfrak{Z}  ^k \oplus  \mathfrak{Z}  ^{ k \perp _W } =
  \mathfrak{Z}  ^k \cap \bigl( \overline{ \mathfrak{B}  ^k } \oplus
  \mathfrak{B} ^{k \perp} \bigr) \oplus \mathfrak{Z}  ^{ k \perp _W } =
  \overline{ \mathfrak{B}  ^k } \oplus \mathfrak{H}  ^k \oplus
  \mathfrak{Z}  ^{ k \perp _W } ,
\end{equation*}
where the final expression is known as the \emph{weak Hodge
  decomposition}.  
For the domain complex $ \left( V,
  \mathrm{d} \right) $, the spaces $ \mathfrak{Z} ^k $, $ \mathfrak{B}
^k $, and $ \mathfrak{H} ^k $ are the same as for $ \left( W ,
  \mathrm{d} \right) $, and consequently we get the decomposition
\begin{equation*}
  V ^k = \overline{ \mathfrak{B}  ^k } \oplus \mathfrak{H}  ^k \oplus
  \mathfrak{Z} ^{ k \perp _V } ,
\end{equation*}
where $ \mathfrak{Z} ^{ k \perp _V } = \mathfrak{Z} ^{ k \perp _W }
\cap V ^k $.  In particular, if $ \left( W, \mathrm{d} \right) $ is a
closed Hilbert complex, then the image $ \mathfrak{B} ^k $ is a closed
subspace, so we have the \emph{strong Hodge decomposition}
\begin{equation*}
  W ^k = \mathfrak{B}  ^k \oplus \mathfrak{H}  ^k \oplus
  \mathfrak{Z}  ^{ k \perp _W } ,
\end{equation*}
and likewise for the domain complex,
\begin{equation*}
  V ^k = \mathfrak{B}  ^k \oplus \mathfrak{H}  ^k \oplus
  \mathfrak{Z} ^{ k \perp _V } .
\end{equation*}
From here on, following the notation of \citet{ArFaWi2010}, we will
simply write $ \mathfrak{Z} ^{ k \perp} $ in place of $ \mathfrak{Z}
^{ k \perp _V } $ when there can be no confusion.

\begin{lemma}[abstract Poincar\'e inequality]
  \label{lem:poincareInequality}
  If $ \left( V, \mathrm{d} \right) $ is a bounded, closed Hilbert
  complex, then there exists a constant $ c _P $ such that
  \begin{equation*}
    \left\lVert v \right\rVert _V \leq c _P {\lVert \mathrm{d} ^k v
      \rVert} _V , \quad \forall v \in \mathfrak{Z}  ^{ k \perp } .
  \end{equation*}
\end{lemma}

\begin{proof}
  The map $ \mathrm{d} ^k $ is a bounded bijection from $ \mathfrak{Z}
  ^{ k \perp } $ to $ \mathfrak{B} ^{ k + 1 } $, which are both closed
  subspaces, so the result follows immediately by applying Banach's
  bounded inverse theorem.
\end{proof}

\begin{corollary}
  If $ \left( V, \mathrm{d} \right) $ is the domain complex of a
  closed Hilbert complex $ \left( W, \mathrm{d} \right) $, then
  \begin{equation*}
    \left\lVert v \right\rVert _V \leq c _P {\lVert \mathrm{d} ^k v
      \rVert} , \quad \forall v \in \mathfrak{Z}  ^{ k \perp } .
  \end{equation*}
\end{corollary}

We close this subsection by defining the dual complex of a Hilbert
complex, and recalling how the Hodge decomposition can be interpreted in
terms of this complex.

\begin{definition}
  Given a Hilbert complex $ \left( W, \mathrm{d} \right) $, the
  \emph{dual complex} $ \left( W ^\ast , \mathrm{d} ^\ast \right) $
  consists of the spaces $ W _k ^\ast = W ^k $, and adjoint operators
  $ \mathrm{d} _k ^\ast = \left( \mathrm{d} ^{ k - 1 } \right) ^\ast
  \colon V _k ^\ast \subset W _k ^\ast \rightarrow V _{ k - 1 } ^\ast
  \subset W _{ k - 1 } ^\ast $.
\begin{equation*}
  \xymatrix{
    \cdots & \ar[l]  V _{k-1} ^\ast & \ar[l] _-{ \mathrm{d} _k ^\ast }
    V _k ^\ast & \ar[l] _-{ \mathrm{d} _{ k + 1 } ^\ast }  V _{ k
      + 1 } ^\ast & \ar[l] \cdots 
  } 
\end{equation*}
\end{definition}

\begin{remark}
  Since the arrows in the dual complex point in the opposite
  direction, this is a Hilbert chain complex rather than a cochain
  complex.  (The chain property $ \mathrm{d} _k ^\ast \circ \mathrm{d}
  _{ k + 1 } ^\ast = 0 $ follows immediately from the cochain property
  $ \mathrm{d} ^k \circ \mathrm{d} ^{ k - 1 } = 0 $.) Accordingly, we
  can define the \emph{$k$-cycles} $ \mathfrak{Z} _k ^\ast = \ker
  \mathrm{d} _k ^\ast = \mathfrak{B} ^{ k \perp _W } $ and \emph{$ k
    $-boundaries} $ \mathfrak{B} _k ^\ast = \mathrm{d} _{ k + 1 }
  ^\ast V ^\ast _k $.  The $k$th harmonic space can then be rewritten
  as $ \mathfrak{H} ^k = \mathfrak{Z} ^k \cap \mathfrak{Z} _k ^\ast $;
  we also have $ \mathfrak{Z} ^k = \mathfrak{B} _k ^{\smash{\ast \perp
      _W }} $, and thus $ \mathfrak{Z} ^{k \perp _W } = \overline{
    \mathfrak{B} _k ^\ast } $.  Therefore, the weak Hodge
  decomposition can be written as
  \begin{equation*}
    W ^k = \overline{ \mathfrak{B}  ^k } \oplus \mathfrak{H}  ^k \oplus
    \overline{ \mathfrak{B}  _k ^\ast } ,
  \end{equation*}
  and in particular, for a closed Hilbert complex, the strong Hodge
  decomposition now becomes
  \begin{equation*}
    W ^k = \mathfrak{B}  ^k \oplus \mathfrak{H}  ^k \oplus
    \mathfrak{B}  _k ^\ast .
  \end{equation*}
\end{remark}

\subsection{The abstract Hodge Laplacian and mixed variational
  problem}
The \emph{abstract Hodge Laplacian} is the operator $ L = \mathrm{d}
\mathrm{d} ^\ast + \mathrm{d} ^\ast \mathrm{d} $, which is an
unbounded operator $ W ^k \rightarrow W ^k $ with domain
\begin{equation*}
  D _L = \left\{ u \in V ^k \cap V _k ^\ast \;\middle\vert\;
    \mathrm{d} u \in V _{ k + 1 } ^\ast ,\ \mathrm{d} ^\ast u \in V ^{
      k - 1 } \right\} .
\end{equation*}
This is a generalization of the Hodge Laplacian for differential
forms, which itself is a generalization of the usual scalar and vector
Laplacian operators on domains in $ \mathbb{R}^n $ (as well as of the
Laplace--Beltrami operator on Riemannian manifolds).

If $u \in D _L $ solves $ L u = f $, then it satisfies the variational
principle
\begin{equation*}
  \left\langle \mathrm{d} u , \mathrm{d} v \right\rangle +
  \left\langle \mathrm{d} ^\ast u , \mathrm{d} ^\ast v \right\rangle =
  \left\langle f, v \right\rangle , \quad \forall v \in V ^k \cap V _k
  ^\ast .
\end{equation*}
However, as noted by \citet{ArFaWi2010}, there are some difficulties
in using this variational principle for a finite element
approximation.  First, it may be difficult to construct finite
elements for the space $ V ^k \cap V _k ^\ast $.  A second concern is
the well-posedness of the problem.  If we take any harmonic test
function $ v \in \mathfrak{H} ^k $, then the left-hand side vanishes,
so $ \left\langle f, v \right\rangle = 0 $; hence, a solution only
exists if $ f \perp \mathfrak{H} ^k $.  Furthermore, for any $ q \in
\mathfrak{H} ^k = \mathfrak{Z} ^k \cap \mathfrak{Z} _k ^\ast $, we
have $ \mathrm{d} q = 0 $ and $ \mathrm{d} ^\ast q = 0 $; therefore,
if $u$ is a solution, then so is $ u + q $.

To avoid these existence and uniqueness issues, one instead defines
the following mixed variational problem: Find $ \left( \sigma, u, p
\right) \in V ^{ k - 1 } \times V ^k \times \mathfrak{H} ^k $
satisfying
\begin{equation}
  \label{eqn:mixedProblem}
  \begin{alignedat}{2}
    \left\langle \sigma , \tau \right\rangle - \left\langle u ,
      \mathrm{d} \tau \right\rangle &= 0, &\quad \forall \tau &\in V
    ^{ k - 1 } ,\\
    \left\langle \mathrm{d} \sigma, v \right\rangle + \left\langle
      \mathrm{d} u, \mathrm{d} v \right\rangle + \left\langle p, v
    \right\rangle &= \left\langle f, v \right\rangle , &\quad \forall
    v &\in V ^k ,\\
    \left\langle u, q \right\rangle &= 0 , &\quad \forall q &\in
    \mathfrak{H} ^k .
  \end{alignedat}
\end{equation}
Here, the first equation implies that $ \sigma = \mathrm{d} ^\ast u $,
which weakly enforces the condition $ u \in V ^k \cap V _k ^\ast
$. Next, the second equation incorporates the additional term $
\left\langle p, v \right\rangle $, which allows for solutions to exist
even when $f \not\perp \mathfrak{H} ^k $.  Finally, the third equation
fixes the issue of non-uniqueness by requiring $ u \perp \mathfrak{H}
^k $.  The following result establishes the well-posedness of the
problem \eqref{eqn:mixedProblem}.

\begin{theorem}[\citet{ArFaWi2010}, Theorem 3.1]
  Let $ \left( W, \mathrm{d} \right) $ be a closed Hilbert complex
  with domain complex $ \left( V, \mathrm{d} \right) $.  The mixed
  formulation of the abstract Hodge Laplacian is well-posed.  That is,
  for any $ f \in W ^k $, there exists a unique $ \left( \sigma, u, p
  \right) \in V ^{ k - 1 } \times V ^k \times \mathfrak{H} ^k $
  satisfying \eqref{eqn:mixedProblem}.  Moreover,
  \begin{equation*}
    \left\lVert \sigma \right\rVert _V + \left\lVert u \right\rVert _V
    + \left\lVert p \right\rVert \leq c \left\lVert f \right\rVert ,
  \end{equation*}
  where $c$ is a constant depending only on the Poincar\'e constant $
  c _P $ in \autoref{lem:poincareInequality}.
\end{theorem}

To prove this, they observe that \eqref{eqn:mixedProblem} can be
rewritten as a standard variational problem---i.e., one having the
form $ B \left( x, y \right) = F (y) $---on the space $ V ^{ k - 1 }
\times V ^k \times \mathfrak{H} ^k $, by defining the bilinear form
\begin{equation*}
  B \left( \sigma  , u , p ; \tau , v , q \right)
  = \left\langle \sigma  , \tau  \right\rangle  - \left\langle u
     , \mathrm{d} \tau \right\rangle  \\
  + \left\langle \mathrm{d} \sigma  , v  \right\rangle  +
  \left\langle \mathrm{d} u  , \mathrm{d} v  \right\rangle  +
  \left\langle p  , v  \right\rangle  - \left\langle u  , q 
  \right\rangle 
\end{equation*} 
and the functional $ F \left( \tau, v, q \right) = \left\langle f, v
\right\rangle $.  The well-posedness of the mixed problem then follows
by establishing the inf-sup condition for the bilinear form $B \left(
  \cdot , \cdot \right) $ \citep[Theorem 3.2]{ArFaWi2010}, which shows
that it defines a linear homeomorphism.  This well-posedness result
implies the existence of a bounded solution operator $ K \colon W ^k
\rightarrow W ^k $ defined by $ K f = u $.

\subsection{Approximation by a subcomplex}

In order to obtain approximate numerical solutions to the mixed
variational problem \eqref{eqn:mixedProblem}, \citet{ArFaWi2010}
suppose that one is given a (finite-dimensional) subcomplex $ V _h
\subset V $ of the domain complex: that is, $ V _h ^k \subset V ^k $
is a Hilbert subspace for each $k$, and the inclusion mapping $ i _h
\colon V _h \hookrightarrow V $ is a morphism of Hilbert complexes.
By analogy with the Galerkin method, one can then consider the mixed
variational problem on the subcomplex: Find $ \left( \sigma _h , u _h
  , p _h \right) \in V _h ^{ k - 1 } \times V _h ^k \times
\mathfrak{H} _h ^k $ satisfying
\begin{equation}
  \label{eqn:subcomplexProblem}
  \begin{alignedat}{2}
    \left\langle \sigma _h , \tau \right\rangle - \left\langle u _h ,
      \mathrm{d} \tau \right\rangle &= 0, &\quad \forall \tau &\in V
    _h
    ^{ k - 1 } ,\\
    \left\langle \mathrm{d} \sigma _h , v \right\rangle + \left\langle
      \mathrm{d} u _h , \mathrm{d} v \right\rangle + \left\langle p _h
      , v \right\rangle &= \left\langle f, v \right\rangle , &\quad
    \forall
    v &\in V _h  ^k ,\\
    \left\langle u _h , q \right\rangle &= 0 , &\quad \forall q &\in
    \mathfrak{H} _h ^k .
  \end{alignedat}
\end{equation}

For the error analysis of this method, one more crucial assumption
must be made: that there exists some Hilbert complex ``projection'' $
\pi _h \colon V \rightarrow V _h $.  We put ``projection'' in quotes
because this need not be the actual orthogonal projection $ i _h ^\ast
$ with respect to the inner product; indeed, that projection is not
generally a morphism of Hilbert complexes, since it may not commute
with the differentials.  However, the map $ \pi _h $ is $V$-bounded,
surjective, and idempotent.  It follows, then, that although it does
not satisfy the optimality property of the orthogonal projection, it
does still satisfy a \emph{quasi-optimality} property, since
\begin{equation*}
  \left\lVert u - \pi _h u  \right\rVert _V = \inf _{ v \in V _h }
  \left\lVert \left( I - \pi _h \right) \left( u - v \right)
  \right\rVert _V \leq \left\lVert I - \pi _h \right\rVert \inf _{ v
    \in V _h } \left\lVert u - v \right\rVert _V ,
\end{equation*}
where the first step follows from the idempotence of $ \pi _h $, i.e.,
$ \left( I - \pi _h \right) v = 0 $ for all $ v \in V _h $.  With this
framework in place, the following error estimate can be established.

\begin{theorem}[\citet{ArFaWi2010}, Theorem 3.9]
  \label{thm:subcomplex}
  Let $ \left( V _h , \mathrm{d} \right) $ be a family of subcomplexes
  of the domain complex $ \left( V, \mathrm{d} \right) $ of a closed
  Hilbert complex, parametrized by $h$ and admitting uniformly
  $V$-bounded cochain projections, and let $ \left( \sigma, u, p
  \right) \in V ^{ k - 1 } \times V ^k \times \mathfrak{H} ^k $ be the
  solution of \eqref{eqn:mixedProblem} and $ \left( \sigma _h , u _h ,
    p _h \right) \in V _h ^{ k - 1 } \times V _h ^k \times
  \mathfrak{H} _h ^k $ the solution of problem
  \eqref{eqn:subcomplexProblem}.  Then
  \begin{multline*}
    \left\lVert \sigma - \sigma _h \right\rVert _V + \left\lVert u - u
      _h \right\rVert _V + \left\lVert p - p _h \right\rVert \\
    \leq C \bigl( \inf _{ \tau \in V _h ^{ k - 1 } } \left\lVert
      \sigma - \tau \right\rVert _V + \inf _{ v \in V _h ^k }
    \left\lVert u - v \right\rVert _V + \inf _{ q \in V _h ^k }
    \left\lVert p - q \right\rVert _V + \mu \inf _{ v \in V _h ^k }
    \left\lVert P _{ \mathfrak{B} } u - v \right\rVert _V \bigr) ,
  \end{multline*}
  where $ \mu = \mu _h ^k = \displaystyle\sup _{ \substack{r \in
      \mathfrak{H} ^k \\ \left\lVert r \right\rVert = 1 } }
  \left\lVert \left( I - \pi _h ^k \right) r \right\rVert $.
\end{theorem}

Therefore, if $ V _h $ is pointwise approximating, in the sense that $
\inf _{ v \in V _h } \left\lVert u - v \right\rVert \rightarrow 0 $ as
$ h \rightarrow 0 $ for every $ u \in V $, then the numerical solution
converges to the exact solution.

\subsection{Improved error estimates}
\label{sec:AFWimproved}
Finally, it can be shown that one can establish improved estimates in
the $W$-norm, subject to a ``compactness property.''  The Hilbert
complex $ \left( W, \mathrm{d} \right) $ is said to have the
\emph{compactness property} if $ V ^k \cap V _k ^\ast $ is a dense
subset of $ W ^k $, and if the inclusion $ \mathcal{I} \colon V ^k
\cap V _k ^\ast \hookrightarrow W ^k $ is compact.  Furthermore,
assume that the family of projections $ \pi _h $ is uniformly
$W$-bounded (rather than merely $V$-bounded) with respect to $h$.
These properties hold for many important examples---notably the $ L ^2
$-de~Rham complex of differential forms---and allows for an abstract
generalization of duality-based $ L ^2 $ estimates (i.e., the
Aubin--Nitsche trick) to the mixed variational problem.

The compactness of the inclusion implies that $K$ is also compact, so
one may define the coefficients
\begin{gather*}
  \delta = \delta _h ^k = \left\lVert \left( I - \pi _h \right) K
  \right\rVert _{ \mathcal{L} \left( W ^k , W ^k \right) } , \qquad
  \mu = \mu _h ^k = \left\lVert \left( I - \pi _h \right) P _{
      \mathfrak{H} } \right\rVert _{ \mathcal{L} \left( W ^k , W ^k
    \right) },\\
  \eta = \eta _h ^k = \max _{ j = 0,1} \left\{ \left\lVert \left( I -
        \pi _h \right) \mathrm{d} K \right\rVert _{ \mathcal{L} \left(
        W ^{ k - j }, W ^{ k - j + 1 } \right) } , \left\lVert \left(
        I - \pi _h \right) \mathrm{d} ^\ast K \right\rVert _{
      \mathcal{L} \left( W ^{ k + j }, W ^{ k + j - 1 } \right) }
  \right\} ,
\end{gather*}
each of which vanishes in the limit as $ h \rightarrow 0 $.  Next, let
us denote best approximation in the $W$-norm by
\begin{equation*}
  E (w) = \inf _{ v \in V _h ^k } \left\lVert w - v \right\rVert ,
  \quad w \in W ^k .
\end{equation*} 
Then the improved estimates are stated in the following theorem.

\begin{theorem}[\citet{ArFaWi2010}, Theorem 3.11]
  \label{thm:improved}
  Let $ \left( V, \mathrm{d} \right) $ be the domain complex of a
  closed Hilbert complex $ \left( W, \mathrm{d} \right) $ satisfying
  the compactness property, and let $ \left( V _h , \mathrm{d} \right)
  $ be a family of subcomplexes parametrized by $h$ and admitting
  uniformly $W$-bounded cochain projections.  Let $ \left( \sigma, u,
    p \right) \in V ^{ k - 1 } \times V ^k \times \mathfrak{H} ^k $ be
  the solution of \eqref{eqn:mixedProblem} and $ \left( \sigma _h , u
    _h , p _h \right) \in V _h ^{ k - 1 } \times V _h ^k \times
  \mathfrak{H} _h ^k $ the solution of problem
  \eqref{eqn:subcomplexProblem}.  Then for some constant $C$
  independent of $h$ and $ \left( \sigma, u, p \right) $, we have
\begin{align*}
  \left\lVert \mathrm{d} \left( \sigma - \sigma _h \right)
  \right\rVert &\leq C E \left( \mathrm{d} \sigma \right), \\
  \left\lVert \sigma - \sigma _h  \right\rVert &\leq C \left[ E
    (\sigma) + \eta E \left( \mathrm{d} \sigma \right) \right], \\
  \left\lVert p - p _h  \right\rVert &\leq C \left[ E (p) + \mu
    E \left( \mathrm{d} \sigma \right) \right], \\
  \left\lVert \mathrm{d} \left( u - u _h  \right) \right\rVert
  &\leq C \left( E \left( \mathrm{d} u \right) + \eta \left[ E \left(
        \mathrm{d} \sigma \right) + E (p) \right] \right) ,\\
  \left\lVert u - u _h \right\rVert &\leq C \bigl( E (u) +
  \eta \left[ E \left( \mathrm{d} u \right) + E (\sigma) \right] \\
  &\qquad + \left( \eta ^2 + \delta \right) \left[ E \left( \mathrm{d}
      \sigma \right) + E (p) \right] + \mu E \left( P _{ \mathfrak{B}
    } u \right) \bigr) .
\end{align*}
\end{theorem}

For typical applications to the de~Rham complex, $ V _h ^k $ consists
of piecewise polynomials defined on a mesh.  In this case, the order
of these coefficients is given by $ \eta = O (h) $, $ \delta = O
\left( h ^{ \min(2, r+1) } \right) $, and $ \mu = O \left( h ^{ r + 1
  } \right) $, where $r$ is the largest degree of complete polynomials
in $ V _h ^k $ (\citet[p.~312]{ArFaWi2010}).

\subsection{Variational crimes}
\label{sec:linearCrimes}
More generally, suppose that the discrete complex $ V _h $ is not
necessarily a subcomplex of $V$, but that we merely have a $W$-bounded
inclusion map $ i _h \colon V _h \hookrightarrow V $, which is a
morphism of Hilbert complexes.  Furthermore, given the $V$-bounded
projection map $ \pi _h \colon V \rightarrow V _h $, we require that $
\pi _h ^k \circ i _h ^k = \operatorname{id} _{ V _h ^k } $ for each
$k$ (which corresponds to the idempotence of $ \pi _h $ when $ i _h $
is simply the inclusion of a subcomplex $ V _h \subset V $).  When $ i
_h $ is unitary---that is, when the discrete inner product satisfies $
\left\langle u _h , v _h \right\rangle _h = \left\langle i _h u _h , i
  _h v _h \right\rangle $ for all $ u _h , v _h \in V _h ^k $---then
this is precisely equivalent to considering the subcomplex $ i _h V _h
\subset V $.  However, if $ i _h $ is not necessarily unitary, we have
a generalized version of the discrete variational problem
\eqref{eqn:subcomplexProblem}, stated as follows: Find $ \left( \sigma
  _h , u _h , p _h \right) \in V _h ^{ k - 1 } \times V _h ^k \times
\mathfrak{H} _h ^k $ satisfying
\begin{equation}
  \label{eqn:discreteProblem}
  \begin{alignedat}{2}
    \left\langle \sigma _h , \tau _h \right\rangle _h - \left\langle u
      _h , \mathrm{d} _h \tau _h \right\rangle _h &= 0, &\quad \forall
    \tau _h &\in V _h ^{ k - 1 } ,\\
    \left\langle \mathrm{d} _h \sigma _h , v _h \right\rangle _h +
    \left\langle \mathrm{d} _h u _h , \mathrm{d} _h v _h \right\rangle
    _h + \left\langle p _h , v _h \right\rangle _h &= \left\langle f
      _h , v _h \right\rangle _h , &\quad \forall v _h
    &\in V _h ^k , \\
    \left\langle u _h , q _h \right\rangle _h &= 0 , &\quad \forall q
    _h &\in \mathfrak{H} _h ^k .
  \end{alignedat}
\end{equation} 
The additional error in this generalized discretization, relative to
the problem on the subcomplex $ i _h V _h \subset V $, arises from two
particular variational crimes: one resulting from the failure of $ i
_h $ to be unitary, and another resulting from the difference between
$ f _h $ and $ i _h ^\ast f $.

In \citet{HoSt2010}, we analyze this additional error by introducing a
modified problem on $ V _h $, which is equivalent to the subcomplex
problem on $ i _h V _h \subset V $.  Define $ J _h = i _h ^\ast i _h
$, so that for any $ u _h , v _h \in W _h $, we have $ \left\langle i
  _h u _h , i _h v _h \right\rangle = \left\langle i _h ^\ast i _h u
  _h , v _h \right\rangle _h = \left\langle J _h u _h , v _h
\right\rangle _h $.  (The norm $ \left\lVert I - J _h \right\rVert $,
therefore, quantifies the failure of $ i _h $ to be unitary.)  This
defines a modified inner product on $ W _h ^k $, leading to a modified
Hodge decomposition $ W _h ^k = \mathfrak{B} _h ^k \oplus \mathfrak{H}
_h ^{\prime k} \oplus \mathfrak{Z} _h ^{\smash{ k \perp\prime_W }} $,
where
\begin{equation*}
  \mathfrak{H} _h ^{\prime k} = \left\{ z \in \mathfrak{Z}  _h
    ^k \;\middle\vert\; i _h z \perp i _h \mathfrak{B}  _h ^k
  \right\}, \qquad \mathfrak{Z} _h ^{\smash{k \perp \prime _W
    }} = \left\{ v \in W _h ^k \;\middle\vert\; i _h v \perp i _h
    \mathfrak{Z}  _h ^k   \right\} .
\end{equation*}
Then the subcomplex problem is equivalent to the following mixed
problem: Find $ \left( \sigma _h ^\prime , u _h ^\prime , p _h ^\prime
\right) \in V _h ^{ k - 1 } \times V _h ^k \times \mathfrak{H} _h ^{
  \prime k } $ satisfying
\begin{equation}
  \label{eqn:modifiedProblem}
  \begin{alignedat}{2}
    \left\langle J _h \sigma _h ^\prime , \tau _h \right\rangle _h -
    \left\langle J _h u _h ^\prime , \mathrm{d} _h \tau _h
    \right\rangle _h &= 0, &\quad \forall
    \tau _h &\in V _h ^{ k - 1 } ,\\
    \left\langle J _h \mathrm{d} _h \sigma _h ^\prime , v _h
    \right\rangle _h + \left\langle J _h \mathrm{d} _h u _h ^\prime ,
      \mathrm{d} _h v _h \right\rangle _h + \left\langle J _h p
      ^\prime _h , v _h \right\rangle _h &= \left\langle i _h ^\ast f
      , v _h \right\rangle _h , &\quad \forall v _h
    &\in V _h ^k , \\
    \left\langle J _h u _h ^\prime , q _h ^\prime \right\rangle _h &=
    0 , &\quad \forall q _h ^\prime &\in \mathfrak{H} _h ^{\prime k }
    .
  \end{alignedat}
\end{equation}
The additional error, between the generalized problem
\eqref{eqn:discreteProblem} and the subcomplex problem
\eqref{eqn:modifiedProblem}, is estimated in the following theorem.

\begin{theorem}[\citet{HoSt2010}, Theorem 3.9]
  \label{thm:HoSt2010-3.9}
  Suppose that $ \left( \sigma _h , u _h , p _h \right) \in V _h ^{ k
    - 1 } \times V _h ^k \times \mathfrak{H} _h ^k $ is a solution to
  \eqref{eqn:discreteProblem} and $ \left( \sigma _h ^\prime , u _h
    ^\prime , p _h ^\prime \right) \in V _h ^{ k - 1 } \times V _h ^k
  \times \mathfrak{H} _h ^{ \prime k } $ is a solution to
  \eqref{eqn:modifiedProblem}.  Then
  \begin{equation*}
    \left\lVert \sigma _h - \sigma _h ^\prime \right\rVert _{ V _h } +
    \left\lVert u _h - u _h ^\prime \right\rVert _{ V _h } +
    \left\lVert p _h - p _h ^\prime \right\rVert _h \leq C \left(
      \left\lVert f _h - i _h ^\ast f \right\rVert _h + \left\lVert I
        - J _h \right\rVert \left\lVert f \right\rVert \right) .
  \end{equation*}
\end{theorem}

Using the triangle inequality, together with the previously stated
result of Arnold, Falk, and Winther (\autoref{thm:subcomplex}) for the
subcomplex problem, we immediately get the following corollary.

\begin{corollary}[\citet{HoSt2010}, Corollary 3.10]
  \label{cor:HoSt2010-3.10}
  If $ \left( \sigma , u , p \right) \in V ^{ k - 1 } \times V ^k
  \times \mathfrak{H} ^k $ is a solution to \eqref{eqn:mixedProblem}
  and $ \left( \sigma _h , u _h , p _h \right) \in V _h ^{ k - 1 }
  \times V _h ^k \times \mathfrak{H} _h ^k $ is a solution to
  \eqref{eqn:discreteProblem}, then
    \begin{multline*}
      \left\lVert \sigma - i _h \sigma _h \right\rVert _V +
      \left\lVert u - i _h u _h \right\rVert _V + \left\lVert p - i _h
        p _h \right\rVert \\
      \leq C \bigl( \inf _{ \tau \in i _h V _h ^{ k - 1 } }
      \left\lVert \sigma - \tau \right\rVert _V + \inf _{ v \in i _h V
        _h ^k } \left\lVert u - v \right\rVert _V + \inf _{ q \in i _h
        V _h ^k } \left\lVert p - q \right\rVert _V + \mu \inf _{ v
        \in i _h V _h
        ^k } \left\lVert P _{ \mathfrak{B} } u - v \right\rVert _V \\
      + \left\lVert f _h - i _h ^\ast f \right\rVert _h + \left\lVert
        I - J _h \right\rVert \left\lVert f \right\rVert \bigr) ,
  \end{multline*}
  where $\mu$ is defined as in \autoref{thm:subcomplex}.
\end{corollary}

This raises the question of how to choose $ f _h \in W _h
$ such that $ f _h \rightarrow i _h ^\ast f $ as $ h \rightarrow 0 $.
While $ f _h = i _h ^\ast f $ would be the ideal choice, of course, it
may be difficult to compute the inner product on $W$, and hence to
compute the adjoint $ i _h ^\ast $.  The following result shows that,
if $ \Pi _h \colon W ^k \rightarrow W _h ^k $ is any bounded linear
projection (i.e., satisfying $ \Pi _h \circ i _h ^k = \mathrm{id} _{ W
  _h ^k } $), then choosing $ f _h = \Pi _h f $ is sufficient to
control this term.

\begin{theorem}[\citet{HoSt2010}, Theorem 3.11]
  \label{thm:HoSt2010-3.11}
  If $ \Pi _h \colon W ^k \rightarrow W _h ^k $ is a family of linear
  projections, bounded uniformly with respect to $h$, then we have the
  inequality
  \begin{equation*}
    \left\lVert \Pi _h f - i _h ^\ast f \right\rVert _h \leq C \bigl(
    \left\lVert I - J _h \right\rVert \left\lVert f \right\rVert +
    \displaystyle\inf _{ \phi  \in i _h W _h ^k } \left\lVert f - \phi 
    \right\rVert \bigr) .
  \end{equation*} 
\end{theorem}

Thus, if the family of discrete complexes satisfies the
``well-approximating'' condition, and if $ \left\lVert I - J _h
\right\rVert \rightarrow 0 $ as $ h \rightarrow 0 $, then it follows
that the generalized discrete solution converges to the continuous
solution.

\section{Semilinear mixed problems}
\label{sec:semilinear}

\subsection{An alternative approach to the linear problem}
In this subsection, we introduce a slightly modified approach to the
linear problem, which will be more useful in the nonlinear analysis to
follow.

Consider the linear operator $ \mathbf{L} = L \oplus P _{ \mathfrak{H}
} \colon D _L \rightarrow W ^k $.  Given any $ \mathbf{u} \in {D} _L
$, we can orthogonally decompose $ \mathbf{u} = u + p $, where $ p = P
_{ \mathfrak{H} } \mathbf{u} $ and $ u = \mathbf{u} - p $.  Therefore,
\begin{equation*}
  \mathbf{L} \mathbf{u} = L \mathbf{u}  + P _{ \mathfrak{H}
  } \mathbf{u}  = L u + p ,
\end{equation*}
so given some $ f \in W ^k $, solving $ L u + p = f $ is equivalent to
solving $ \mathbf{L} \mathbf{u} = f $.  Furthermore, if we define the
solution operator $ \mathbf{K} = K \oplus P _{ \mathfrak{H} } $, it
follows that
\begin{equation*}
  \mathbf{K} f = K f + P _{ \mathfrak{H}  } f = u + p = \mathbf{u} ,
\end{equation*}
so $\mathbf{K}$ is in fact the inverse of $\mathbf{L}$.  Thus,
$\mathbf{L}$ and $\mathbf{K}$ establish a bijection between $ D _L $
and $ W ^k $.  Effectively, by adding $ P _{ \mathfrak{H} } $ to each
of the operators $L$ and $K$, we have managed to remove their kernel $
\mathfrak{H} ^k $.

This approach also sheds new light on the well-posedness of the linear
problem.  If $ \mathbf{u} $ is a solution
to $ \mathbf{L} \mathbf{u} = f $, then it satisfies the variational
problem: Find $ \mathbf{u} \in V ^k \cap V _k ^\ast $ such that
\begin{equation}
  \label{eqn:unmixedProblem}
  \left\langle \mathrm{d} ^\ast \mathbf{u} , \mathrm{d} ^\ast v
  \right\rangle + \left\langle \mathrm{d} \mathbf{u} , \mathrm{d} v
  \right\rangle + \left\langle P _{ \mathfrak{H}  } \mathbf{u} , P _{
      \mathfrak{H}  } v \right\rangle = \left\langle f, v
  \right\rangle , \quad \forall v \in V ^k \cap V _k ^\ast .
\end{equation}
In fact, the left-hand side is precisely the inner product $
\left\langle \mathbf{u}, v \right\rangle _{ V \cap V ^\ast } $, which
is equivalent to the usual intersection inner product obtained by
adding the inner products for $V$ and $ V ^\ast $
(\citet[p.~312]{ArFaWi2010}).  Hence, by the Riesz representation
theorem, a unique solution $ \mathbf{u} = \mathbf{K} f $ exists, and
moreover $\mathbf{K}$ is bounded.  In particular, this variational
formulation also illustrates that $\mathbf{K}$ is the adjoint to the
bounded inclusion $ \mathcal{I} \colon V ^k \cap V ^\ast _k
\hookrightarrow W ^k $, with respect to this $ \left\langle \cdot ,
  \cdot \right\rangle _{ V \cap V ^\ast } $ inner product, and thus $
\mathbf{K} $ must be bounded as well.

\begin{remark}
  While the solutions to the two variational problems
  \eqref{eqn:mixedProblem} and \eqref{eqn:unmixedProblem} are
  equivalent, the mixed formulation is still preferable for
  implementing finite element methods, since one may not have
  efficient finite elements for the space $ V ^k \cap V _k ^\ast $.
  We emphasize that this alternative approach is introduced primarily
  to make the analysis of semilinear problems more convenient.
\end{remark}

\subsection{Semilinear problems and the abstract Hammerstein equation}
Given some $ f \in W ^k $, we are interested in the semilinear problem
of finding $\mathbf{u}$, such that
\begin{equation}
  \label{eqn:semilinearOperator}
  \mathbf{L} \mathbf{u} + F \mathbf{u} = f,
\end{equation}
where $ F \colon V ^k \rightarrow W ^k $ is some nonlinear operator.
Extending the argument from the linear case, it follows that this
operator equation is equivalent to the mixed variational problem: Find
$ \left( \sigma ,u , p \right) \in V ^{k-1} \times V ^k \times
\mathfrak{H} ^k $ satisfying
\begin{equation}
  \label{eqn:mixedSemilinear}
  \begin{alignedat}{2}
    \left\langle \sigma, \tau \right\rangle - \left\langle u ,
      \mathrm{d} \tau \right\rangle &= 0, \quad &\forall \tau &\in V
    ^{ k - 1 } \\
    \left\langle \mathrm{d} \sigma , v \right\rangle + \left\langle
      \mathrm{d} u , \mathrm{d} v \right\rangle + \left\langle p, v
    \right\rangle + \left\langle F \left( u + p \right) , v
    \right\rangle &= \left\langle f, v \right\rangle , \quad &\forall
    v &\in V ^k ,\\
    \left\langle u, q \right\rangle &= 0 , \quad &\forall q &\in
    \mathfrak{H} ^k .
  \end{alignedat}
\end{equation}
In the special case where $F = 0 $, this simply reduces to the linear
problem.

Using the solution operator $\mathbf{K}$, the equation
\eqref{eqn:semilinearOperator} is also equivalent to
\begin{equation}
  \label{eqn:hammerstein}
  \mathbf{u} + \mathbf{K} F \mathbf{u} = \mathbf{K} f .
\end{equation}
Equations having this general form are called \emph{abstract
  Hammerstein equations}, and are of particular interest in nonlinear
functional analysis (cf.~\citet{Zeidler1990b}).  This formulation,
which notably appeared in the seminal papers of
\citet{Amann1969,Amann1969a} and \citet{BrGu1969}, generalizes certain
nonlinear integral equations, called Hammerstein integral
equations. (In the context of integral equations, the operator
$\mathbf{K}$ corresponds to the kernel operator, or Green's operator.)

\subsection{Well-posedness of the semilinear problem} Before we
establish the well-posedness of the abstract Hammerstein equation
\eqref{eqn:hammerstein}, it is necessary to define some special
properties that a nonlinear operator may have.

\begin{definition}
  The operator $ A \colon W ^k \rightarrow W ^k $ is said to be
  \emph{monotone} if, for all $ u, v \in W ^k $, it satisfies $
  \left\langle A u - A v, u - v \right\rangle \geq 0 $.  It is called
  \emph{strictly monotone} if $ \left\langle A u - A v, u - v
  \right\rangle > 0 $ whenever $ u \neq v $, and \emph{strongly
    monotone} if there exists a constant $ c > 0 $ such that $
  \left\langle A u - A v, u - v \right\rangle \geq c \left\lVert u - v
  \right\rVert ^2 $.
\end{definition}

\begin{definition}
  The operator $ A \colon W ^k \rightarrow W ^k $ is said to be
  \emph{hemicontinuous} if the real function $ t \mapsto \left\langle
    A \left( u + t v \right) , w \right\rangle $ is continuous on $
  \left[ 0, 1 \right] $ for all $ u , v , w \in W ^k $.
\end{definition}

\begin{theorem}
  \label{thm:wellPosed}
  If $F$ is monotone and hemicontinuous, then the semilinear problem
  \eqref{eqn:semilinearOperator} has a unique solution.  Moreover, the
  problem is well-posed: given two functionals $ f $ and $ f ^\prime
  $, the respective solutions $ \mathbf{u} $ and $ \mathbf{u} ^\prime
  $ satisfy the Lipschitz continuity estimate $ \left\lVert \mathbf{u}
    - \mathbf{u} ^\prime \right\rVert _{ V \cap V ^\ast } \leq
  \left\lVert \mathbf{K} \right\rVert \left\lVert f - f ^\prime
  \right\rVert $.
\end{theorem}

The existence/uniqueness portion of the proof is an adaptation of a
standard argument for Hammerstein equations, when the kernel operator
is symmetric and monotone on some real, separable Hilbert space
(cf.~\citet[p.~618]{Zeidler1990b}).

\begin{proof}
  Let us define the operator $ A = I + \mathbf{K} F $ on $ V ^k \cap V
  _k ^\ast $, so that the abstract Hammerstein equation
  \eqref{eqn:hammerstein} can be written as $ A \mathbf{u} =
  \mathbf{K} f $.  Since $F$ is hemicontinuous, it follows that $A$ is
  also hemicontinuous.  Moreover, $A$ is strongly monotone with
  constant $ c = 1 $, since for any $ \mathbf{u} , \mathbf{u} ^\prime
  \in V ^k \cap V _k ^\ast $, we have
  \begin{align*}
    \left\langle A \mathbf{u} - A \mathbf{u} ^\prime , \mathbf{u} -
      \mathbf{u} ^\prime \right\rangle _{ V \cap V ^\ast } &=
    \left\lVert \mathbf{u} - \mathbf{u} ^\prime \right\rVert _{ V \cap
      V ^\ast } ^2 + \left\langle \mathbf{K} \left( F \mathbf{u} - F
        \mathbf{u} ^\prime \right) , \mathbf{u} - \mathbf{u} ^\prime
    \right\rangle _{ V \cap V ^\ast }\\
    &= \left\lVert \mathbf{u} - \mathbf{u} ^\prime \right\rVert _{ V
      \cap V ^\ast } ^2 + \left\langle F \mathbf{u} - F \mathbf{u}
      ^\prime , \mathbf{u} - \mathbf{u} ^\prime \right\rangle \\
    &\geq \left\lVert \mathbf{u} - \mathbf{u} ^\prime \right\rVert ^2
    _{ V \cap V ^\ast } ,
  \end{align*} 
  where the last line follows from the monotonicity of $F$.
  Therefore, since $A$ is hemicontinuous and strongly monotone, the
  Browder--Minty theorem \citep{Browder1963,Minty1962} implies that it
  has a Lipschitz continuous inverse $ A ^{-1} $ with Lipschitz
  constant $ c ^{-1} = 1 $.  Hence, there exist unique solutions $
  \mathbf{u} = A ^{-1} \mathbf{K} f $ and $ \mathbf{u} ^\prime = A
  ^{-1} \mathbf{K} f ^\prime $.  Finally, by the fact that $ A ^{-1} $
  is nonexpansive, these solutions satisfy
  \begin{equation*}
    \left\lVert \mathbf{u} - \mathbf{u} ^\prime \right\rVert _{ V \cap
      V ^\ast } \leq \left\lVert \mathbf{K} f - \mathbf{K} f ^\prime
    \right\rVert _{ V \cap V ^\ast } \leq \left\lVert \mathbf{K}
    \right\rVert \left\lVert f - f ^\prime \right\rVert ,
  \end{equation*} 
  which completes the proof.
\end{proof}

\subsection{Solution estimate for the mixed formulation} Now that we
have established the well-posedness of the semilinear problem
\eqref{eqn:semilinearOperator}, we can use the \emph{linear} solution
theory, as developed by \citet{ArFaWi2010}, to develop a similar
estimate for the mixed formulation. This requires placing slightly
stronger conditions on the nonlinear operator $F$.  In particular, we
require $F$ to be Lipschitz continuous with respect to the $V$-norm:
that is, there exists a constant $C$ such that
\begin{equation*}
  \left\lVert F \mathbf{u} - F \mathbf{u} ^\prime \right\rVert  \leq C
  \left\lVert \mathbf{u} - \mathbf{u} ^\prime \right\rVert _V ,
\end{equation*} 
for all $ \mathbf{u} , \mathbf{u} ^\prime \in V ^k $.  (Later, in
\autoref{sec:localLipschitz}, we will see how this condition can be
relaxed in case $F$ is only locally Lipschitz.)

\begin{theorem}
  \label{thm:mixedWellPosed}
  If $F$ is monotone and Lipschitz continuous with respect to the
  $V$-norm, then the mixed semilinear problem
  \eqref{eqn:mixedSemilinear} has a unique solution $ \left( \sigma,
    u, p \right) $.  Moreover, the problem is well-posed: given two
  functionals $ f $ and $ f ^\prime $, the respective solutions $
  \left( \sigma, u, p \right) $ and $ \left( \sigma ^\prime , u
    ^\prime , p ^\prime \right) $ satisfy the Lipschitz continuity
  estimate
  \begin{equation*}
    \left\lVert \sigma - \sigma ^\prime \right\rVert _V + \left\lVert
      u - u ^\prime \right\rVert _V + \left\lVert p - p ^\prime
    \right\rVert \leq C \left\lVert f - f ^\prime \right\rVert ,
  \end{equation*} 
  where the constant $C$ depends only on the Poincar\'e constant $ c
  _P $ and on the Lipschitz constant of $F$.
\end{theorem}

\begin{proof}
  If $\mathbf{u}$ is a solution of the semilinear problem $ \mathbf{L}
  \mathbf{u} + F \mathbf{u} = f $, then it is also a solution of the
  linear problem $ \mathbf{L} \mathbf{u} = g $, where $ g = f - F
  \mathbf{u} $.  Therefore, $ \left( \sigma, u , p \right) \in V ^{ k
    - 1 } \times V ^k \times \mathfrak{H} ^k $ is the unique solution
  of the mixed linear problem with functional $g$, and hence of the
  mixed semilinear problem \eqref{eqn:mixedSemilinear}.

  Now, suppose that $ \mathbf{u} ^\prime $ is the solution to $
  \mathbf{L} \mathbf{u} ^\prime + F \mathbf{u} ^\prime = f ^\prime $,
  and hence to the linear problem $ \mathbf{L} \mathbf{u} ^\prime = g
  ^\prime = f ^\prime - F \mathbf{u} ^\prime $.  Define $ \overline{
    \mathbf{u} } = \mathbf{u} - \mathbf{u} ^\prime $ and $ \overline{
    g } = g - g ^\prime $; subtracting the two linear equations $
  \mathbf{L} \mathbf{u} = g $ and $ \mathbf{L} \mathbf{u} ^\prime = g
  ^\prime $, it follows that $ \mathbf{L} \overline{ \mathbf{u} } =
  \overline{ g } $.  Therefore, $ \left( \overline{ \sigma } ,
    \overline{ u } , \overline{ p } \right) = \left( \sigma - \sigma
    ^\prime , u - u ^\prime , p - p ^\prime \right) $ satisfies the
  mixed linear problem with functional $ \overline{ g } $, so by the
  well-posedness of the mixed linear problem, we have
  \begin{equation*}
    \left\lVert \overline{ \sigma } \right\rVert _V + \left\lVert
      \overline{ u } \right\rVert _V + \left\lVert \overline{ p }
    \right\rVert \leq c \left\lVert \overline{ g } \right\rVert ,
  \end{equation*} 
  where $c$ depends only on the Poincar\'e constant $ c _P $.  Next,
  the right-hand side can be estimated by
  \begin{multline*}
    \left\lVert \overline{ g } \right\rVert \leq \left\lVert f - f
      ^\prime \right\rVert + \left\lVert F \mathbf{u} - F \mathbf{u}
      ^\prime \right\rVert \\
    \leq \left\lVert f - f ^\prime \right\rVert + C \left\lVert
      \mathbf{u} - \mathbf{u} ^\prime \right\rVert _V \leq \left\lVert
      f - f ^\prime \right\rVert + C \left\lVert \mathbf{u} -
      \mathbf{u} ^\prime \right\rVert _{ V \cap V ^\ast } ,
  \end{multline*}
  using the Lipschitz property of $F$.  Finally, applying the
  previously-obtained estimate $ \left\lVert \mathbf{u} - \mathbf{u}
    ^\prime \right\rVert _{ V \cap V ^\ast } \leq \left\lVert
    \mathbf{K} \right\rVert \left\lVert f - f ^\prime \right\rVert $,
  we get $ \left\lVert \overline{ g } \right\rVert \leq C \left\lVert
    f - f ^\prime \right\rVert $, so finally
  \begin{equation*}
    \left\lVert \sigma - \sigma ^\prime \right\rVert _V + \left\lVert
      u - u ^\prime \right\rVert _V + \left\lVert p - p ^\prime
    \right\rVert \leq C \left\lVert f - f ^\prime \right\rVert ,
  \end{equation*}
  which completes the proof.
\end{proof}

\begin{remark}
  Note that, in the linear case where $ F = 0 $, we can take $ f
  ^\prime = 0 $ so that $ \left( \sigma ^\prime , u ^\prime , p
    ^\prime \right) = 0 $.  Then, since $ g = f $ and $ g ^\prime = f
  ^\prime = 0 $, we simply recover the usual linear estimate $
  \left\lVert \sigma \right\rVert _V + \left\lVert u \right\rVert _V +
  \left\lVert p \right\rVert \leq c \left\lVert f \right\rVert $.
\end{remark}

\section{Approximation theory and numerical analysis}
\label{sec:approx}

\subsection{The discrete semilinear problem} To set up the discrete
semilinear problem, and develop the subsequent convergence results, we
begin by assuming the same conditions as in the linear case. Namely,
suppose that $ V _h \subset V $ is a Hilbert subcomplex, equipped with
a bounded cochain projection $ \pi _h \colon V \rightarrow V _h $.
Let $ K _h \colon W _h ^k \rightarrow W _h ^k $ be the discrete
solution operator for the linear problem, taking $ P _h f \mapsto u _h
$.  As with the continuous problem, we define a new solution operator
$ \mathbf{K} _h = K _h \oplus P _{ \mathfrak{H} _h } $ and consider
the discrete Hammerstein equation
\begin{equation*}
  \mathbf{u} _h + \mathbf{K} _h P _h F \mathbf{u} _h = \mathbf{K} _h P
  _h f .
\end{equation*}
Note that this is \emph{not} simply the Galerkin problem for the
original Hammerstein operator equation \eqref{eqn:hammerstein}, since
$ \mathbf{K} _h $ is not just a projection of $\mathbf{K}$ onto the
discrete space; in particular, we generally have $ \mathfrak{H} _h ^k
\not\subset \mathfrak{H} ^k $.

This is precisely the abstract Hammerstein equation on the discrete
Hilbert complex $ V _h $, in the sense of the previous section.
Therefore, there exists a unique solution $ \mathbf{u} _h $, and the
discrete solution operator $ P _h f \mapsto \mathbf{u} _h $, $ P _h f
^\prime \mapsto \mathbf{u} _h ^\prime $, satisfies the Lipschitz
condition
\begin{equation*}
  \left\lVert \mathbf{u} _h - \mathbf{u} _h ^\prime \right\rVert _{ V
    _h \cap V _h ^\ast } \leq
  \left\lVert \mathbf{K} _h \right\rVert \left\lVert P _h \left( f - f
      ^\prime \right) \right\rVert \leq \left\lVert \mathbf{K} _h
  \right\rVert \left\lVert f - f ^\prime \right\rVert .
\end{equation*} 
Equivalently, this gives a solution to the discrete mixed variational
problem: Find $ \left( \sigma _h , u _h , p _h \right) \in V _h ^{ k
  -1 } \times V _h ^k \times \mathfrak{H} _h ^k $ satisfying
\begin{equation}
  \label{eqn:mixedSubcomplexSemilinear}
  \begin{alignedat}{2}
    \left\langle \sigma _h , \tau \right\rangle - \left\langle u _h ,
      \mathrm{d} \tau \right\rangle &= 0, &\quad \forall \tau &\in V
    _h
    ^{ k - 1 } ,\\
    \left\langle \mathrm{d} \sigma _h , v \right\rangle + \left\langle
      \mathrm{d} u _h , \mathrm{d} v \right\rangle + \left\langle p _h
      , v \right\rangle + \left\langle F \left( u _h + p _h \right) ,
      v \right\rangle &= \left\langle f, v \right\rangle , &\quad
    \forall v &\in V _h  ^k ,\\
    \left\langle u _h , q \right\rangle &= 0 , &\quad \forall q &\in
    \mathfrak{H} _h ^k .
  \end{alignedat}
\end{equation}
If $F$ is Lipschitz, then we also obtain an estimate for the mixed
solution,
\begin{equation*}
  \left\lVert \sigma _h - \sigma _h ^\prime \right\rVert _V +
  \left\lVert u _h - u _h ^\prime \right\rVert _V + \left\lVert p _h -
    p _h ^\prime \right\rVert \leq C _h \left\lVert f - f ^\prime
  \right\rVert .
\end{equation*} 
Finally, we remark that when $ V _h $ is a family of subcomplexes
parametrized by $h$, and the projections $ \pi _h \colon V
\rightarrow V _h $ are bounded uniformly with respect to $h$, then the
constants in these estimates may also be bounded independently of $h$.

\subsection{Convergence of the discrete solution}
We now estimate the error in approximating the solution of the mixed
semilinear problem \eqref{eqn:mixedSemilinear} by that for the
discrete problem \eqref{eqn:mixedSubcomplexSemilinear}.  Despite the
introduction of nonlinearity, we obtain the same quasi-optimal
estimate as in \autoref{thm:subcomplex} for the linear problem.

\begin{theorem}
   \label{thm:errorEstimate}
   Let $ \left( V _h , \mathrm{d} \right) $ be a family of
   subcomplexes of the domain complex $ \left( V, \mathrm{d} \right) $
   of a closed Hilbert complex, parametrized by $h$ and admitting
   uniformly $V$-bounded cochain projections, and let $ \left( \sigma,
     u, p \right) \in V ^{ k - 1 } \times V ^k \times \mathfrak{H} ^k
   $ be the solution of \eqref{eqn:mixedSemilinear} and $ \left(
     \sigma _h , u _h , p _h \right) \in V _h ^{ k - 1 } \times V _h
   ^k \times \mathfrak{H} _h ^k $ the solution of problem
   \eqref{eqn:mixedSubcomplexSemilinear}.  Then, assuming the operator
   $F$ is Lipschitz with respect to the $V$-norm, we have the estimate
  \begin{multline*}
    \left\lVert \sigma - \sigma _h \right\rVert _V + \left\lVert u - u
      _h \right\rVert _V + \left\lVert p - p _h \right\rVert \\
    \leq C \bigl( \inf _{ \tau \in V _h ^{ k - 1 } } \left\lVert
      \sigma - \tau \right\rVert _V + \inf _{ v \in V _h ^k }
    \left\lVert u - v \right\rVert _V + \inf _{ q \in V _h ^k }
    \left\lVert p - q \right\rVert _V + \mu \inf _{ v \in V _h ^k }
    \left\lVert P _{ \mathfrak{B} } u - v \right\rVert _V \bigr) ,
  \end{multline*}
  where $\mu$ is defined as in \autoref{thm:subcomplex}, and where the
  constant $C$ depends only on the Poincar\'e constant $ c _P $ and
  the Lipschitz constant of $F$.
\end{theorem}

\begin{proof}
  Recall that, since $ \left( \sigma, u, p \right) $ solves the
  semilinear problem for the functional $f$, it also solves the linear
  problem for the functional $ g = f - F \left( u + p \right) $.  Let
  $ \left( \sigma _h ^\prime , u _h ^\prime , p _h ^\prime \right) \in
  V _h ^{ k -1 } \times V _h ^k \times \mathfrak{H} _h ^k $ be the
  solution to the corresponding discrete linear problem for $g$.  By
  \autoref{thm:subcomplex}, this satisfies the error estimate
  \begin{multline*}
    \left\lVert \sigma - \sigma _h ^\prime \right\rVert _V +
    \left\lVert u - u _h ^\prime \right\rVert _V + \left\lVert p - p
      _h ^\prime \right\rVert \\
    \leq C \bigl( \inf _{ \tau \in V _h ^{ k - 1 } } \left\lVert
      \sigma - \tau \right\rVert _V + \inf _{ v \in V _h ^k }
    \left\lVert u - v \right\rVert _V + \inf _{ q \in V _h ^k }
    \left\lVert p - q \right\rVert _V + \mu \inf _{ v \in V _h ^k }
    \left\lVert P _{ \mathfrak{B} } u - v \right\rVert _V \bigr) .
  \end{multline*}
  Next, observe that $ \left( \sigma _h ^\prime , u _h ^\prime , p _h
    ^\prime \right) $ is also a solution of the discrete semilinear
  problem with functional $ f ^\prime = f - F \left( u + p \right) + F
  \left( u _h ^\prime + p _h ^\prime \right) $, since we can just add
  $ F \left( u _h ^\prime + p _h ^\prime \right) $ to both sides of
  the equation.  However, since the discrete solution operator is
  Lipschitz, we have
  \begin{align*}
    \left\lVert \sigma _h - \sigma _h ^\prime \right\rVert _V +
    \left\lVert u _h - u _h ^\prime \right\rVert _V + \left\lVert p _h
      - p _h ^\prime \right\rVert &\leq C \left\lVert f - f ^\prime
    \right\rVert \\
    &= C \left\lVert F \left( u + p \right) - F \left( u _h ^\prime +
        p _h ^\prime \right) \right\rVert .
  \end{align*}  
  Furthermore, since $F$ is also Lipschitz,
  \begin{equation*}
    \left\lVert F \left( u + p \right) - F \left( u _h ^\prime +
        p _h ^\prime \right) \right\rVert \leq   C \left( \left\lVert u - u _h
        ^\prime \right\rVert _V + \left\lVert p - p _h ^\prime
      \right\rVert \right) ,
  \end{equation*} 
  which implies
  \begin{equation*}
    \left\lVert \sigma _h - \sigma _h ^\prime \right\rVert _V +
    \left\lVert u _h - u _h ^\prime \right\rVert _V + \left\lVert p _h
      - p _h ^\prime \right\rVert \leq  C \left(     \left\lVert
        \sigma - \sigma _h ^\prime \right\rVert _V +  \left\lVert u -
        u _h ^\prime \right\rVert _V + \left\lVert p - p  _h ^\prime
      \right\rVert \right) .
  \end{equation*} 
  An application of the triangle inequality completes the proof.
\end{proof}

As in the linear case, this implies that if $ V _h $ is pointwise
approximating in $V$ as $ h \rightarrow 0 $, then $ \left( \sigma _h ,
  u _h , p _h \right) \rightarrow \left( \sigma, u, p \right) $.
Moreover, the rate of convergence for this semilinear problem is the
same as that for the linear problem.

\subsection{Improved estimates} We now establish improved estimates
for the semilinear problem, subject to the compactness property
introduced in \autoref{sec:AFWimproved}.

\begin{theorem}
  Let $ \left( V, \mathrm{d} \right) $ be the domain complex of a
  closed Hilbert complex $ \left( W, \mathrm{d} \right) $ satisfying
  the compactness property, and let $ \left( V _h , \mathrm{d} \right)
  $ be a family of subcomplexes parametrized by $h$ and admitting
  uniformly $W$-bounded cochain projections.  Let $ \left( \sigma, u,
    p \right) \in V ^{ k - 1 } \times V ^k \times \mathfrak{H} ^k $ be
  the solution of \eqref{eqn:mixedSemilinear} and $ \left( \sigma _h ,
    u _h , p _h \right) \in V _h ^{ k - 1 } \times V _h ^k \times
  \mathfrak{H} _h ^k $ the solution of problem
  \eqref{eqn:mixedSubcomplexSemilinear}, and assume that the operator
  $F$ is Lipschitz.  Then for some constant $C$ independent of $h$ and
  $ \left( \sigma, u, p \right) $, we have
\begin{align*}
  \left\lVert \mathrm{d} \left( \sigma - \sigma _h \right)
  \right\rVert &\leq C \bigl[ E \left( \mathrm{d} \sigma \right) + E
  (u) + E \left( \mathrm{d} u \right) + E (p) \\
  &\qquad + \eta E (\sigma) +
  \mu E \left( P _{ \mathfrak{B}  } u \right) \bigr] \\
  \left\lVert \sigma - \sigma _h \right\rVert &\leq C \bigl[ E
  (\sigma) + E (u) + E \left( \mathrm{d} u \right) + E (p) \\
  &\qquad + \left( \eta + \delta + \mu \right) E \left( \mathrm{d}
    \sigma \right) + \mu E \left( P _{ \mathfrak{B} } u \right) \bigr]
  \\
  \left\lVert u - u _h \right\rVert _V + \left\lVert p - p _h
  \right\rVert &\leq C \bigl( E (u) + E \left( \mathrm{d} u \right) +
  E (p) \\
  &\qquad + \eta \left[ E (\sigma) + E \left( \mathrm{d} \sigma
    \right) \right] + \left( \delta + \mu \right) E \left( \mathrm{d}
    \sigma \right) + \mu E \left( P _{ \mathfrak{B} } u \right) \bigr) .
\end{align*}
\end{theorem}

\begin{proof}
  As before, let $ \left( \sigma _h ^\prime , u _h ^\prime , p _h
    ^\prime \right) \in V _h ^{ k - 1 } \times V _h ^k \times
  \mathfrak{H} _h ^k $ be the solution to the discrete linear problem
  with right-hand side functional $ g = f - F \left( u + p \right) $.
  Then \autoref{thm:improved} gives the improved estimates
  \begin{align*}
    \left\lVert \mathrm{d} \left( \sigma - \sigma _h ^\prime \right)
    \right\rVert &\leq C E \left( \mathrm{d} \sigma \right), \\
    \left\lVert \sigma - \sigma _h ^\prime \right\rVert &\leq C \left[ E
      (\sigma) + \eta E \left( \mathrm{d} \sigma \right) \right], \\
    \left\lVert p - p _h ^\prime \right\rVert &\leq C \left[ E (p) + \mu
      E \left( \mathrm{d} \sigma \right) \right], \\
    \left\lVert \mathrm{d} \left( u - u _h ^\prime \right) \right\rVert
    &\leq C \left( E \left( \mathrm{d} u \right) + \eta \left[ E \left(
          \mathrm{d} \sigma \right) + E (p) \right] \right) ,\\
    \left\lVert u - u _h ^\prime \right\rVert &\leq C \bigl( E (u) +
    \eta \left[ E \left( \mathrm{d} u \right) + E (\sigma) \right] \\
    &\qquad + \left( \eta ^2 + \delta \right) \left[ E \left( \mathrm{d}
        \sigma \right) + E (p) \right] + \mu E \left( P _{ \mathfrak{B}
      } u \right) \bigr) .
  \end{align*}
  However, in the proof of \autoref{thm:errorEstimate}, we saw that
  each of the terms $ \left\lVert \mathrm{d} \left( \sigma _h - \sigma
      _h ^\prime \right) \right\rVert $, $ \left\lVert \sigma _h -
    \sigma _h ^\prime \right\rVert $, and $ \left\lVert u _h - u _h
    ^\prime \right\rVert _V + \left\lVert p _h - p _h ^\prime
  \right\rVert $ is controlled by
  \begin{multline*}
    \left\lVert \sigma _h - \sigma _h ^\prime \right\rVert _V +
    \left\lVert u _h - u _h ^\prime \right\rVert _V + \left\lVert p _h
      - p _h ^\prime \right\rVert\\
    \begin{aligned}
      &\leq C \left( \left\lVert u - u _h ^\prime \right\rVert _V +
        \left\lVert p - p _h ^\prime \right\rVert
      \right) \\
      &\leq C \bigl( E (u) + E \left( \mathrm{d} u \right) + E (p) \\
      &\qquad + \eta \left[ E (\sigma) + E \left( \mathrm{d} \sigma
        \right) \right] + \left( \delta + \mu \right) E \left(
        \mathrm{d} \sigma \right) + \mu E \left( P _{ \mathfrak{B} } u
      \right) \bigr).
    \end{aligned}
  \end{multline*}
  Applying the triangle inequality and eliminating higher-order terms,
  the result follows immediately.
\end{proof}

\subsection{Semilinear variational crimes}

As first discussed in \autoref{sec:linearCrimes}, suppose now that $ V
_h $ is not necessarily a subcomplex of $V$, and let $ i _h \colon V
_h \hookrightarrow V $ and $ \pi _h \colon V \rightarrow V _h $ be the
$W$-bounded inclusion and $V$-bounded projection morphisms,
respectively, satisfying $ \pi _h \circ i _h = \operatorname{id} _{ V
  _h } $.  Given a discrete functional $ f _h \in W _h ^k $ and a
discrete nonlinear operator $ F _h \colon V _h ^k \rightarrow W _h ^k
$, we wish to approximate the continuous variational problem
\eqref{eqn:mixedSemilinear} by the discrete problem: Find $ \left(
  \sigma _h , u _h , p _h \right) \in V _h ^{ k -1 } \times V _h ^k
\times \mathfrak{H} _h ^k $ satisfying
\begin{equation}
  \label{eqn:discreteSemilinear}
  \begin{alignedat}{2}
    \left\langle \sigma _h , \tau _h \right\rangle _h - \left\langle u
      _h , \mathrm{d} _h \tau _h \right\rangle _h &= 0, &\quad \forall
    \tau _h &\in V _h ^{ k - 1 } ,\\
    \left\langle \mathrm{d} _h \sigma _h , v _h \right\rangle _h +
    \left\langle \mathrm{d} _h u _h , \mathrm{d} _h v _h
    \right\rangle _h + \left\langle p _h , v _h \right\rangle _h \quad
    \\
    + \left\langle F _h \left( u _h + p _h \right) , v _h
    \right\rangle _h &= \left\langle f _h , v _h \right\rangle _h ,
    &\quad \forall v _h
    &\in V _h ^k , \\
    \left\langle u _h , q _h \right\rangle _h &= 0 , &\quad \forall q
    _h &\in \mathfrak{H} _h ^k .
  \end{alignedat}
\end{equation} 
For the following error estimate, we define the projection map $ P _{
  V _h } \colon V \rightarrow V _h $ so that $ i _h P _{ V _h } v $ is
the $V$-orthogonal projection of $ v $ onto the subcomplex $ i _h V _h
\subset V $.

\begin{theorem}
  Let $ \left( \sigma , u , p \right) \in V ^{ k - 1 } \times V ^k
  \times \mathfrak{H} ^k $ be the solution to
  \eqref{eqn:mixedSemilinear} and $ \left( \sigma _h , u _h , p _h
  \right) \in V _h ^{ k - 1 } \times V _h ^k \times \mathfrak{H} _h ^k
  $ be the solution to \eqref{eqn:discreteSemilinear}.  If $ F _h $ is
  Lipschitz, and its constant is uniformly bounded in $h$, then
    \begin{multline*}
      \left\lVert \sigma - i _h \sigma _h \right\rVert _V +
      \left\lVert u - i _h u _h \right\rVert _V + \left\lVert
        p - i _h p _h  \right\rVert \\
      \leq C \bigl( \inf _{ \tau \in i _h V _h ^{ k - 1 } }
      \left\lVert \sigma - \tau \right\rVert _V + \inf _{ v \in i _h V
        _h ^k } \left\lVert u - v \right\rVert _V + \inf _{ q \in i _h
        V _h ^k } \left\lVert p - q \right\rVert _V + \mu \inf _{ v
        \in i _h V _h
        ^k } \left\lVert P _{ \mathfrak{B} } u - v \right\rVert _V \\
      + \left\lVert i _h ^\ast \left( f - F \left( u + p \right)
        \right) - \left( f _h - F _h P _{V_h} \left( u + p \right)
        \right) \right\rVert _h + \left\lVert I - J _h \right\rVert
      \left\lVert f - F \left( u + p \right) \right\rVert \bigr) ,
  \end{multline*}
  where $\mu$ is defined as in \autoref{thm:subcomplex}.
\end{theorem}

\begin{proof}
  Suppose $ \left( \sigma _h ^\prime , u _h ^\prime , p _h ^\prime
  \right) \in V _h ^{ k - 1 } \times V _h ^k \times \mathfrak{H} _h ^k
  $ is the solution to the discrete linear problem with right-hand
  side functional $ i _h ^\ast g = i _h ^\ast \left( f - F \left( u +
      p \right) \right) $.  Then, applying
  \autoref{cor:HoSt2010-3.10}, we have
    \begin{multline*}
      \left\lVert \sigma - i _h \sigma _h ^\prime \right\rVert _V +
      \left\lVert u - i _h u _h ^\prime \right\rVert _V + \left\lVert
        p - i _h p _h ^\prime \right\rVert \\
      \leq C \bigl( \inf _{ \tau \in i _h V _h ^{ k - 1 } }
      \left\lVert \sigma - \tau \right\rVert _V + \inf _{ v \in i _h V
        _h ^k } \left\lVert u - v \right\rVert _V + \inf _{ q \in i _h
        V _h ^k } \left\lVert p - q \right\rVert _V + \mu \inf _{ v
        \in i _h V _h
        ^k } \left\lVert P _{ \mathfrak{B} } u - v \right\rVert _V \\
      + \left\lVert I - J _h \right\rVert \left\lVert f - F \left( u + p
        \right) \right\rVert \bigr) .
  \end{multline*}
  Next, observe that $ \left( \sigma _h ^\prime , u _h ^\prime , p _h
    ^\prime \right) $ also solves the discrete semilinear problem with
  right-hand side functional $ f _h ^\prime = i _h ^\ast \left( f - F
    \left( u + p \right) \right) + F _h \left( u _h ^\prime + p _h
    ^\prime \right) $.  Therefore, since the discrete solution
  operator is Lipschitz, we obtain
  \begin{multline*}
    \left\lVert \sigma _h - \sigma _h ^\prime \right\rVert _{ V _h } +
    \left\lVert u _h - u _h ^\prime \right\rVert _{ V _h } +
    \left\lVert p _h - p _h ^\prime \right\rVert _h \\
    \begin{aligned}
      &\leq C \left\lVert i _h ^\ast \left( f - F \left( u + p \right)
        \right) - \left( f _h - F _h \left( u _h ^\prime + p _h
            ^\prime \right) \right) \right\rVert _h \\
      &\leq C \left\lVert i _h ^\ast \left( f - F \left( u + p \right)
        \right) - \left( f _h - F _h P _{ V _h } \left( u + p \right)
        \right) \right\rVert _h \\
      &\qquad + \left\lVert F _h P _{ V _h } \left( u + p \right) - F
        _h \left( u _h ^\prime + p _h ^\prime \right) \right\rVert _h
      .
    \end{aligned}
  \end{multline*} 
  Applying the Lipschitz property of $ F _h $ to the last term of this
  expression,
  \begin{align*}
    \left\lVert F _h P _{ V _h } \left( u + p \right) - F _h \left( u
        _h ^\prime + p _h ^\prime \right) \right\rVert _h &\leq C
    \left( \left\lVert P _{ V _h } u - u _h ^\prime \right\rVert _{ V
        _h } + \left\lVert P _{ V _h } p - p _h ^\prime \right\rVert
      _{ V _h } \right) \\
    &= C \left( \left\lVert P _{ V _h } \left( u - i _h u _h ^\prime
        \right) \right\rVert _{ V _h } + \left\lVert P _{ V _h }
        \left( p - i _h p _h ^\prime \right) \right\rVert _{ V _h }
    \right) \\
    &\leq C \left( \left\lVert u - i _h u _h ^\prime \right\rVert _V +
      \left\lVert p - i _h p _h ^\prime \right\rVert \right),
  \end{align*}
  which we have already controlled.  Hence, an application of the
  triangle inequality completes the proof.
\end{proof}

Clearly, the optimal choice for the functional $ f _h $ and the
operator $ F _h $ would be
\begin{equation*}
  f _h = i _h ^\ast f , \qquad F _h = i _h ^\ast F i _h .
\end{equation*} 
In this case, we would obtain
\begin{multline*}
  \left\lVert i _h ^\ast \left( f - F \left( u + p \right) \right) -
    \left( f _h - F _h P _{ V _h } \left( u + p \right) \right)
  \right\rVert _h \\
  \begin{aligned}
    &= \left\lVert i _h ^\ast \left( F \left( u + p \right) - F i _h P
        _{ V _h } \left( u + p \right) \right) \right\rVert _h \\
    &\leq C \left\lVert \left( I - i _h P _{ V _h } \right) \left( u +
        p \right) \right\rVert _V \\
    &\leq C \left( \inf _{ v \in i _h V _h ^k } \left\lVert u - v
      \right\rVert _V + \inf _{ q \in i _h V _h ^k } \left\lVert p - q
      \right\rVert _V \right) ,
  \end{aligned}
\end{multline*}
which already appears elsewhere in the estimate.  Hence, this choice
of $ f _h $ and $ F _h $ allows the term $ \left\lVert i _h ^\ast
  \left( f - F \left( u + p \right) \right) - \left( f _h - F _h
    \left( P _h u + P _h p \right) \right) \right\rVert _h $ to be
dropped.

However, as noted before, it may not be feasible to take $ f _h = i _h
^\ast f $ or $ F _h = i _h ^\ast F i _h $, since it is often difficult
to compute the adjoint $ i _h ^\ast $ to the inclusion.  Instead,
letting $ \Pi _h \colon W ^k \rightarrow W _h ^k $ be any bounded
linear projection, suppose we choose $ f _h = \Pi _h f $ and $ F _h =
\Pi _h F i _h $, effectively approximating $ i _h ^\ast $ by $ \Pi _h
$.  As in the linear case, this choice will give us good convergence
behavior, contributing an error that is again controlled by other
terms in the error estimate.

\begin{theorem}
  Given a family of linear projections $ \Pi _h \colon W ^k
  \rightarrow W _h ^k $, bounded uniformly with respect to $h$,
  suppose that $ f _h = \Pi _h f $ and $ F _h = \Pi _h F i _h $, where
  $F$ is assumed to be Lipschitz.  Then
  \begin{multline*}
    \left\lVert i _h ^\ast \left( f - F \left( u + p \right) \right) -
      \left( f _h - F _h P _{ V _h } \left( u + p \right) \right)
    \right\rVert _h \leq C \bigl( \left\lVert I - J _h \right\rVert
    \left\lVert f - F
      \left( u + p \right) \right\rVert \\
    + \inf _{ \phi \in i _h W _h ^k } \left\lVert \left( f - F \left(
          u + p \right) \right) - \phi \right\rVert + \inf _{ v \in i
      _h V _h ^k } \left\lVert u - v \right\rVert _V + \inf _{ q \in i
      _h V _h ^k } \left\lVert p - q \right\rVert _V \bigr) .
  \end{multline*}
\end{theorem}

\begin{proof}
  We begin by using the triangle inequality to write
  \begin{multline*}
    \left\lVert i _h ^\ast \left( f - F \left( u + p \right) \right) -
      \Pi _h \left( f - F i _h P _{ V _h } \left( u + p \right)
      \right) \right\rVert _h \\
    \leq \left\lVert \left( i _h ^\ast - \Pi _h \right) \left( f - F
        \left( u + p \right) \right) \right\rVert _h + \left\lVert \Pi
      _h \left( F \left( u + p \right) - F i _h P _{ V _h } \left( u +
          p \right) \right) \right\rVert _h .
  \end{multline*}
  For the first term, we can apply \autoref{thm:HoSt2010-3.11} to
  obtain
  \begin{multline*}
    \left\lVert \left( i _h ^\ast - \Pi _h \right) \left( f - F
        \left( u + p \right) \right) \right\rVert _h \\
    \leq C \bigl( \left\lVert I - J _h \right\rVert \left\lVert f - F
      \left( u + p \right) \right\rVert + \inf _{ \phi \in i _h W _h
      ^k } \left\lVert \left( f - F \left( u + p \right) \right) -
      \phi \right\rVert \bigr) .
  \end{multline*} 
  For the remaining term, we have
  \begin{align*}
    \left\lVert \Pi _h \left( F \left( u + p \right) - F i _h P _{ V
          _h } \left( u + p \right) \right) \right\rVert _h &\leq C
    \left\lVert F \left( u + p \right) - F i _h P _{ V _h } \left( u +
        p \right) \right\rVert \\
    &\leq C \left\lVert \left( I - i _h P _{ V _h } \right) \left( u +
        p \right) \right\rVert _V \\
    &\leq C \left( \inf _{ v \in i _h V _h ^k } \left\lVert u - v
      \right\rVert _V + \inf _{ q \in i _h V _h ^k } \left\lVert p - q
      \right\rVert _V \right) ,
  \end{align*}
which completes the proof.
\end{proof}

Hence, we again get convergence of the discrete solution to the
continuous solution, as long as the discrete complex is
well-approximating and $ \left\lVert I - J _h \right\rVert \rightarrow
0 $ as $ h \rightarrow 0 $.

\subsection{Remarks on relaxing the Lipschitz assumption}
\label{sec:localLipschitz}
Our \emph{a priori} estimates for the mixed semilinear problem
depended, crucially, on the assumption that the monotone operator $F$
was not merely hemicontinuous but Lipschitz.  In many problems of
interest, however, $F$ may be only \emph{locally} Lipschitz: that is,
given $\mathbf{u} \in V ^k $, there exist constants $ C, M > 0 $
(possibly depending on $\mathbf{u}$) such that $ \left\lVert F
  \mathbf{u} - F \mathbf{u} ^\prime \right\rVert \leq C \left\lVert
  \mathbf{u} - \mathbf{u} ^\prime \right\rVert _V $ whenever $
\left\lVert \mathbf{u} - \mathbf{u} ^\prime \right\rVert _V \leq M $.
What can we say about well-posedness and convergence when the
Lipschitz condition is only local rather than global?

Since \autoref{thm:wellPosed} requires only the hemicontinuity of $F$,
we still know that the semilinear problem has a unique solution, and
that it satisfies
\begin{equation*}
  \left\lVert \mathbf{u} - \mathbf{u} ^\prime \right\rVert _{ V \cap V
    ^\ast } \leq \left\lVert \mathbf{K} \right\rVert \left\lVert f - f ^\prime
  \right\rVert .
\end{equation*} 
For the mixed problem, though, all we can show is that
\begin{equation*}
  \left\lVert \sigma - \sigma ^\prime \right\rVert _V + \left\lVert u
    - u ^\prime \right\rVert _V + \left\lVert p - p ^\prime
  \right\rVert \leq C \left( \left\lVert f - f ^\prime \right\rVert +
    \left\lVert F \mathbf{u} - F \mathbf{u} ^\prime \right\rVert
  \right) ,
\end{equation*} 
at which point the proof of \autoref{thm:mixedWellPosed} requires the
Lipschitz condition to continue.  However, if $F$ is locally Lipschitz
at $\mathbf{u}$, then we can still proceed to obtain
\begin{equation*}
  \left\lVert \sigma - \sigma ^\prime \right\rVert _V + \left\lVert u
    - u ^\prime \right\rVert _V + \left\lVert p - p ^\prime
  \right\rVert \leq C \left\lVert f - f ^\prime \right\rVert,
\end{equation*} 
as long as $ \left\lVert f - f ^\prime \right\rVert $ (and therefore $
\left\lVert \mathbf{u} - \mathbf{u} ^\prime \right\rVert _V $) is
sufficiently small.  The same holds true for the well-posedness of the
discrete mixed problem on $ V _h $.

Now, let us observe how this affects the convergence of the discrete
problem.  In the proof of the \emph{a priori} estimate,
\autoref{thm:errorEstimate}, we had 
\begin{equation*}
  \left\lVert f - f ^\prime \right\rVert = \left\lVert F \left( u +
      p \right) - F \left( u _h ^\prime + p _h ^\prime \right)
  \right\rVert ,
\end{equation*} 
where $ \left( \sigma _h ^\prime , u _h ^\prime , p _h ^\prime \right)
$ is the solution to the discrete linear problem with right-hand side
functional $ g = f - F \left( u + p \right) $.  If $ V _h $ is
well-approximating in $V$, then \autoref{thm:subcomplex} imples that,
by taking $h$ sufficiently small, we can get $ \left\lVert f - f
  ^\prime \right\rVert $ to be as small as we want.  Therefore, the
error estimates hold as long as $h$ is sufficiently small.

As an example of how these Lipschitz conditions arise, consider the
following semilinear elliptic problem on a smooth, connected, open
domain $\Omega \subset \mathbb{R}^n$: Find $u \in \mathring{H} ^1
(\Omega)$ such that
\begin{equation}
   \label{eqn:semilinear}
  - \Delta  u + u ^m = f ,
\end{equation}
where $ m \geq 1 $ is an odd integer. Since $ L = - \Delta $ is the
Hodge--Laplace operator for the $ L ^2 $-de~Rham complex when $ k = 0
$, this problem can be expressed within our semilinear framework by
taking $ F u = u ^m $.  While $F$ is monotone (since $m$ is odd), it
does not appear to be globally Lipschitz when $ m > 1 $, since the
inequality
\begin{equation}
   \label{eqn:lipschitz}
  \left\lVert F u - F u ^\prime  \right\rVert _Y \leq C \left\lVert u
    - u ^\prime \right\rVert _X , \quad \forall u , u ^\prime \in X ,
\end{equation} 
cannot be shown to hold for any reasonable choice of the spaces $X$
and $Y$.

However, for semilinear scalar problems where both continuous and
discrete maximum principles are available, it is possible to establish
\emph{a priori} $L^{\infty}$ estimates on the continuous and discrete
solutions.  These estimates ensure that the solutions both lie in an
\emph{order interval} $ \left[ u_-,u_+ \right] \cap
\mathring{H}^1(\Omega) $ within the solution space.  In other words,
if $u$ and $u_h$ are the continuous and discrete solutions of the
semilinear problem~\eqref{eqn:semilinear}, then they satisfy
\begin{equation*} 
u_- \le u, u_h \le u_+.
\end{equation*} 
This pointwise control makes it possible to
establish~\eqref{eqn:lipschitz} in this order interval, where $X =
\mathring{H} ^1 \left( \Omega \right) $ and $Y=L^2(\Omega)$.  This is
precisely the Lipschitz condition that we need to apply the framework
developed in this paper.  In fact, even exponential-type
nonlinearities can be shown to satisfy the
condition~\eqref{eqn:lipschitz} at the continous and discrete
solutions; see, for example, \citep{ChHoXu2007}.  For a discussion of
these and related techniques for semilinear problems,
see~\citep{StHo2011}.

While pointwise control of the continuous solution
to~\eqref{eqn:semilinear} is always available, due to the maximum
principle property of the Laplacian, pointwise control of the discrete
solution is in fact a much more delicate property.  Typically, this
requires placing restrictive angle conditions on the mesh underlying
the finite element space.  In two spatial dimensions, the angle
conditions necessary to preserve the maximum principle property are
achievable with careful mesh generation, even when local mesh
refinement algorithms in are use.  However, in three spatial
dimensions, it is very difficult to satisfy the required angle
conditions, even on quasi-uniform meshes.

Nevertheless, in the case of sub-critical and critical-type polynomial
nonlinearities, it is possible to establish a \emph{local} type of
Lipschitz condition by relying only on pointwise control of the
continuous solution, without requiring pointwise control of the
discrete solution, and thus avoiding the need for mesh conditions
altogether.  For this class of nonlinearities, one can obtain the
following local Lipschitz result.
\begin{theorem}
\label{thm:nonlinear}
Let $\Omega \subset \mathbb{R}^n$ for $n \ge 2$, and assume that
$\|u\|_{L^{\infty}(\Omega)} <\infty$.  Let $F\colon
\mathring{H}^1(\Omega)\to H^{-1}(\Omega)$ be a polynomial in $u$ with
measurable coefficients defined on $\Omega$, and whose polynomial
degree $m$ satisfies $1 \leq m <\infty$ for $n=2$ and $1 \leq m \leq
\overline{m} = (n+2)/(n-2)$ for $n>2$.  Assume also that $u,u' \in
\mathring{H}^1(\Omega)$, and that $\| u -
u'\|_{\mathring{H}^1(\Omega)} \le M$ for some finite constant $M$.
Then
\begin{equation*}
  \| Fu - Fu' \|_{H^{-1}(\Omega)}
  \le C \| u - u' \|_{\mathring{H}^1(\Omega)},
\end{equation*}
where $C=C(\Omega,F,\|u\|_{L^{\infty}(\Omega)},n,m,M)$.
\end{theorem}
\begin{proof}
See~\cite{BHSZ11a}.
\end{proof}

We note that the result in \autoref{thm:nonlinear} has a slightly
different form than that considered above, since $ F \colon \mathring{
  H } ^1 (\Omega) \rightarrow H ^{-1} (\Omega) $ rather than $
\mathring{ H } ^1 (\Omega) \rightarrow L ^2 (\Omega) $.  In the
language of Hilbert complexes, that is, the codomain is given by the
dual to $ V ^k $ instead of $ W ^k $.  However, as remarked by
\citet[p.~305]{ArFaWi2010}, the estimates of finite element exterior
calculus also apply when the data is given weakly as $ f \in \left( V
  ^k \right) ^\ast $, equipped with the sup-norm, and the analysis
does not change substantially from the $ f \in W ^k $ case (although
the solution can no longer be interpreted as giving the Hodge
decomposition of $f$ in a strong sense).  Likewise, the results
presented here for the semilinear problem also extend to the case of
weakly-specified data, since the tools of monotone operator theory and
abstract Hammerstein equations carry over without any significant
modification (other than the appearance of the sup-norm in place of
the $W$-norm, where appropriate).

Finally, many important problems contain nonlinearities satisfying the
assumptions needed to establish continuous and discrete pointwise
control, either by satisfying mesh conditions or by
\autoref{thm:nonlinear}.  In particular, these examples include the
Yamabe problem arising in geometric analysis, and the Hamiltonian
constraint equation in general relativity.  For the three-dimensional
case, the leading nonlinear terms for both of these problems have the
form
\begin{equation*} 
Fu = a u^5 + b u,
\end{equation*} 
where $a, b \in L^{\infty}(\Omega)$.  Since $ m = 5 $ equals the
critical exponent $ \overline{m} = (n+2)/(n-2) $ when $ n = 3 $, the
nonlinearity satisfies the hypotheses of \autoref{thm:nonlinear}.
See~\cite{HoNaTs2009} for the derivation of pointwise bounds for both
problems, using maximum principles.

\section{Conclusion}

In this article, we have extended the abstract Hilbert complex
framework of \citet{ArFaWi2010}, as well as our previous analysis of
variational crimes from \citet{HoSt2010}, to a class of semilinear
mixed variational problems.  Our approach used an equivalent
formulation of these problems as abstract Hammerstein equations,
enabling us to apply the tools of nonlinear functional analysis and
monotone operator theory, and to obtain well-posedness results for
both continuous and discrete semilinear problems.  Additional
continuity assumptions on the nonlinearity yielded a stronger
well-posedness result for mixed problems, as well as \emph{a priori}
error bounds for the discrete solution.  Despite the addition of
nonlinear terms, this result agrees with the quasi-optimal estimate of
\citet{ArFaWi2010} for the linear case, and similarly allows for
improved estimates to be obtained under additional compactness and
continuity assumptions.  Likewise, in extending the variational crimes
analysis in \citep{HoSt2010} to semilinear problems, we obtain
convergence results agreeing with the linear case.  These last results
can also be used to extend the \emph{a priori} estimates for Galerkin
solutions to the Laplace--Beltrami equation on approximate $2$- and
$3$-hypersurfaces, due to \citet{Dziuk1988} and \citet{Demlow2009}, to
the larger class of semilinear problems involving the Hodge Laplacian
on hypersurfaces of arbitrary dimension.

At the conclusion of \citet{HoSt2010}, several open problems are
mentioned, including the extension of the Hilbert complex framework to
more general Banach complexes.  While the Hilbert complex framework
was again sufficient for the analysis of semilinear problems presented
here, Banach spaces become necessary when dealing with more general
nonlinear problems.  Banach complexes appear to lack much of the
crucial structure of Hilbert complexes, particularly the Hodge
decomposition, whose orthogonality depends fundamentally on the
presence of an inner product.  However, if there is additional
structure present in a Banach complex, such as a Gelfand-like triple
structure (e.g., $W \subset H \subset W ^\ast$, where $H$ is a Hilbert
complex), then it may be possible to generalize the approach taken
here.

\begin{acknowledgments}
  M.~H.~was supported in part by NSF DMS/CM Awards 0715146 and 0915220,
  NSF MRI Award 0821816, NSF PHY/PFC Award 0822283, and by DOD/DTRA
  Award HDTRA-09-1-0036.

  A.~S.~was supported in part by NSF DMS/CM Award 0715146 and by NSF
  PHY/PFC Award 0822283, as well as by NIH, HHMI, CTBP, and NBCR.
\end{acknowledgments}

\end{document}